\documentclass[a4paper, reqno]{amsart}

\usepackage[utf8]{inputenc}
\usepackage{mathtools}
\usepackage{graphicx}
\usepackage{array}
\usepackage{color}
\usepackage{hyphenat}

\newtheorem{lemma}{Lemma}[section]
\newtheorem{theorem}[lemma]{Theorem}
\newtheorem{conjecture}[lemma]{Conjecture}
\newtheorem{corollary}[lemma]{Corollary}

\theoremstyle{definition}
\newtheorem{definition}[lemma]{Definition}

\newcommand{\R}{\mathbb{R}}

\newcommand{\Rplus}{\R_+}

\newcommand{\Ham}{\mathbb{H}}

\newcommand{\Jcal}{\mathcal{J}}

\newcommand{\Scal}{\mathcal{S}}

\DeclareMathOperator{\ind}{IND}              
\DeclareMathOperator{\thb}{TH}               
\DeclareMathOperator{\clique}{QIND}          
\DeclareMathOperator{\lthb}{LTH}             
\DeclareMathOperator{\bthb}{TH^\circ}

\newcommand{\borderM}{R}            
\newcommand{\one}{\mathbf{1}}       
\newcommand{\tp}{{\sf T}}           
\DeclareMathOperator{\conv}{conv}   
\DeclareMathOperator{\diag}{diag}   
\DeclareMathOperator{\trace}{tr}    
\DeclareMathOperator{\aut}{Aut}     
\DeclareMathOperator{\iso}{Iso}     
\newcommand{\btheta}{\vartheta^\circ}
\DeclareMathOperator{\hoff}{Hoff}   

\newcommand{\lhof}{(}
\newcommand{\rhof}{)}

\newenvironment{optprob}
{
  \arraycolsep=0pt
  \begin{array}{r@{\ }l@{\quad}l}
}%
{
  \end{array}
}

\newlength\claimlen%
\newcommand{\defi}[1]{\textit{#1}}

\newenvironment{defsymb}[1]
{%
  \par\addvspace{\medskipamount}%
  \leftskip=\parindent%
  \rightskip=\parindent%
  \parindent=0pt%
  \sbox0{#1}%
  \hangindent=\wd0%
  \hangafter=1%
  \vrule height0pt depth0pt width0pt\hbox{#1}\ignorespaces%
}
{\par\addvspace{\medskipamount}}

\title{A recursive theta body for hypergraphs}

\author{Davi Castro-Silva}
\address{D. Castro-Silva, Centrum Wiskunde \& Informatica, Postbus~94079, 1090~GB
  Amsterdam, The Netherlands.}
\email{davi.silva@cwi.nl}

\author{Fernando Mário de Oliveira Filho}
\address{F.M. de Oliveira Filho, Delft Institute of Applied Mathematics,
  Delft University of Technology, Mekelweg~4, 2628~CD Delft, The
  Netherlands.}
\email{F.M.deOliveiraFilho@tudelft.nl}

\author{Lucas Slot}
\address{L. Slot, ETH Zürich, Universitätsstrasse 6, 8092 Zürich, Switzerland.}
\email{lucas.slot@inf.ethz.ch}

\author{Frank Vallentin}
\address{F. Vallentin, Department Mathematik/Informatik, Abteilung
  Mathematik, Universität zu Köln, Weyertal 86--90, 50931 Köln,
  Germany.}
\email{frank.vallentin@uni-koeln.de}

\thanks{This project has received funding from the European Union's Horizon 2020
  research and innovation programme under the Marie Sk\l{}odowska-Curie
  agreement No 764759. The third author was at the Centrum Wiskunde \&
  Informatica, Amsterdam, when this research was carried out.  The fourth author
  is partially supported by the SFB/TRR 191 ``Symplectic Structures in Geometry,
  Algebra and Dynamics'' and by the project ``Spectral bounds in extremal
  discrete geometry'' (project number 414898050), both funded by the DFG}


\date{March 29, 2023}

\begin{document}

\begin{abstract}
  The theta body of a graph, introduced by Grötschel, Lovász, and Schrijver
  in~1986, is a tractable relaxation of the independent-set polytope
  derived from the Lovász theta number. In this paper, we recursively
  extend the theta body, and hence the theta number, to hypergraphs. We
  obtain fundamental properties of this extension and relate it to the
  high-dimensional Hoffman bound of Filmus, Golubev, and Lifshitz.  We
  discuss two applications: triangle-free graphs and Mantel's theorem, and
  bounds on the density of triangle-avoiding sets in the Hamming cube.
\end{abstract}

\maketitle
\markboth{D. Castro-Silva, F.M. de Oliveira Filho, L. Slot, and
  F. Vallentin}{A recursive theta body for hypergraphs}


\section{Introduction}

The theta number of a graph, introduced by Lovász~\cite{Lovasz1979} to
determine the Shannon capacity of the pentagon, is one of the founding
results of semidefinite programming and has inspired numerous developments
in combinatorics (see Grötschel, Lovász, and
Schrijver~\cite[Chapter~9]{GrotschelLS1988} and
Schrijver~\cite[Chapter~67]{Schrijver2003}), coding theory (see
Schrijver~\cite{Schrijver2005}), and discrete geometry (see Oliveira and
Vallentin~\cite{OliveiraV2019}).  It is a graph parameter that provides at
the same time an upper bound for the independence number of a graph and a
lower bound for the chromatic number of its complement, a result known as
Lovász's sandwich theorem.  The theta number also has weighted variants,
and both Lovász's original parameter and its variants can be computed in
polynomial time.  To this day, the only known polynomial-time algorithms to
compute a maximum-weight independent set or a minimum-weight coloring in a
perfect graph compute the weighted theta number as a subroutine.

The sandwich theorem has a geometrical counterpart, the theta body.  The
theta body of a graph~$G = (V, E)$ was introduced by Grötschel, Lovász, and
Schrijver~\cite{GrotschelLS1986}; it is the convex
body~$\thb(G) \subseteq \R^V$ given by the feasible region of the
optimization program defining the theta number.  It contains the
independent-set polytope of~$G$ and is contained in the polytope defined by
the clique inequalities of~$G$.  One can optimize linear functions over the
theta body in polynomial time, that is, the weak optimization problem
over~$\thb(G)$ can be solved in polynomial time.  The theta body provides a
characterization of perfect graphs: $\thb(G)$ is a polytope, and in this
case is exactly the independent-set polytope, if and only if~$G$ is a
perfect graph.
\medbreak

In this paper we extend the definition of the theta body from graphs to
hypergraphs, derive fundamental properties of this extension, and discuss
applications.


\subsection{Independence in hypergraphs}

Let~$H = (V, E)$ be an $r$-uniform hypergraph for some integer~$r \geq 1$,
so~$V$ is a finite set and $E \subseteq \binom{V}{r}$, where $\binom{V}{r}$
denotes the set of $r$-element subsets of $V$.  For $r=2$ this gives the
usual notion of a graph, while the case $r=1$ is somewhat degenerate but
convenient for inductive arguments.  The \defi{complement} of~$H$ is the
$r$-uniform hypergraph~$\overline{H}$ with vertex set~$V$ in which an
$r$-subset~$e$ of~$V$ is an edge if and only if~$e$ is not an edge of~$H$.

A set~$I \subseteq V$ is \defi{independent} in~$H$ if no edge of~$H$ is
contained in~$I$.  Given a weight function~$w \in \R^V$, the
\defi{weighted independence number} of~$H$ is
\[
  \alpha(H, w) = \max\{\, w(I) : \text{$I \subseteq V$ is independent}\,\},
\]
where~$w(I) = \sum_{x \in I} w(x)$.  When~$w = \one$ is the constant-one
function,~$\alpha(H, w)$ is the \defi{independence number} of~$H$, denoted
simply by~$\alpha(H)$.  Computing the independence number of a graph is a
known NP-hard problem~\cite{Karp1972} and computing its hypergraph
counterpart is also NP-hard.

The \defi{independent-set polytope} of~$H$ is the convex hull of
characteristic functions of independent sets of~$H$, namely
\[
  \ind(H) = \conv\{\, \chi_I \in \R^V : \text{$I \subseteq V$ is
    independent}\,\},
\]
where~$\chi_S \in \R^V$ is the characteristic function of~$S \subseteq V$.
The weighted independence number~$\alpha(H, w)$ can be computed by
maximizing~$w^\tp f$ over~$f \in \ind(H)$, and so optimizing over~$\ind(H)$
is an NP-hard problem.

A \defi{clique} of~$H$ is a set~$C \subseteq V$ such that every $r$-subset
of~$C$ is an edge.  Note that cliques of~$H$ are independent sets
of~$\overline{H}$ and vice versa.  Note also that any set with fewer
than~$r$ elements is both a clique and an independent set (the same happens
with graphs: single vertices are both cliques and independent sets).

If~$C$ is a clique of~$H$ and if~$f \in \ind(H)$, then~$f(C) \leq r - 1$.
These valid inequalities for~$\ind(H)$ are called \defi{clique
  inequalities}; they give a relaxation of the independent-set polytope,
namely the polytope
\begin{equation} \label{eq:QIND}
    \clique(H) = \{\, f \in [0, 1]^V : \text{$f(C) \leq r - 1$ for every
    clique~$C \subseteq V$}\, \}.
\end{equation}
Clearly,~$\ind(H) \subseteq \clique(H) \subseteq [0, 1]^V$.  The integer
vectors in $\clique(H)$ are precisely the characteristic functions of
independent sets, and so the integer hull of~$\clique(H)$ is~$\ind(H)$.

Since cliques of~$H$ are independent sets of~$\overline{H}$, the separation
problem over $\clique(H)$ consists of finding a maximum-weight independent
set of~$\overline{H}$, and it is therefore NP-hard.  As a consequence,
optimizing over $\clique(H)$ is NP-hard as well.


\subsection{The theta body of graphs and hypergraphs}

Grötschel, Lovász, and Schrijver~\cite{GrotschelLS1986} defined the theta
body of a graph $G$: a convex relaxation of~$\ind(G)$ stronger
than $\clique(G)$ over which it is possible to optimize a linear function
in polynomial time.

For a symmetric matrix~$A$, write
\[
  \borderM(A) =
  \begin{pmatrix}
    1&a^\tp\\a&A
  \end{pmatrix},
\]
where~$a = \diag A$ is the diagonal of~$A$.  The \defi{theta body} of a
graph~$G = (V, E)$ is
\begin{equation}
  \label{eq:theta-graph}
  \begin{split}
    \thb(G) = \{\, f \in \R^V :\ & \text{there is~$F \in \R^{V \times V}$
      such that $f = \diag F$},\\
    & F(x, y) = 0\quad\text{if~$\{x, y\} \in E$, and}\\
    & \text{$\borderM(F)$ is positive semidefinite}\,\}.
  \end{split}
\end{equation}
(This specific formulation was given by Lovász and
Schrijver~\cite{LovaszS1991}.)  Here and elsewhere, positive semidefinite
matrices are always symmetric.

The theta body is a closed and convex set satisfying
\[
  \ind(G) \subseteq \thb(G) \subseteq \clique(G)
\]
for every graph~$G$; since~$\clique(G)$ is bounded, the theta body is compact.
Moreover, optimizing over the theta body is the same as solving a semidefinite
program, and in this case this can be done to any desired precision in
polynomial time using either the ellipsoid
method~\cite[Chapter~9]{GrotschelLS1988} or the interior-point
method~\cite{KlerkV2016}.

The Lovász theta number of~$G$ for a weight function~$w \in \R^V$ is
obtained by optimizing over the theta body, namely
\[
  \vartheta(G, w) = \max\{\, w^\tp f : f \in \thb(G)\, \};
\]
for~$w = \one$ we recover the theta number as originally defined by
Lovász~\cite{Lovasz1979}, which we denote simply by~$\vartheta(G)$.
Immediately we get
\[
  \alpha(G, w) \leq \vartheta(G, w).
\]

Our aim is to extend the definition of the theta body, and therefore of the
theta number, to $r$-uniform hypergraphs for~$r \geq 3$.  We do so
recursively, and the base of our recursion is~$r = 1$.  By taking this as
the base, we can give uniform proofs without relying on what is known about
the theta body of a graph.  So we will always take~$r = 1$ as the base,
unless this choice would lead us into trouble.

Let~$H = (V, E)$ be an $r$-uniform hypergraph for~$r \geq 2$.
Given~$x \in V$, the \defi{link} of~$x$ in~$H$ is the $(r-1)$-uniform
hypergraph~$H_x$ with vertex set
\[
  V_x = \{\, y \in V\setminus\{x\} : \text{there is~$e \in E$
    containing~$x$ and~$y$}\,\},
\]
in which an $(r-1)$-subset~$e$ of~$V_x$ is an edge if and only
if~$e \cup \{x\}$ is an edge of~$H$.

Given a matrix~$A \in \R^{V \times V}$ and~$x \in V$, let~$A_x \in \R^V$
denote the row of~$A$ indexed by~$x$, that is, $A_x(y) = A(x, y)$.
If~$f\colon V \to \R$ is a function and~$U \subseteq V$ is a set, denote
by~$f[U]$ the restriction of~$f$ to~$U$.

We are now ready to give our main definition.

\begin{definition}%
\label{def:theta-body}%
Let~$H = (V, E)$ be an $r$-uniform hypergraph.  For~$r = 1$, the
\defi{theta body} of~$H$ is~$\thb(H) = \ind(H)$.  For~$r\geq 2$, the
\defi{theta body} of~$H$ is
\[
  \begin{split}
    \thb(H) = \{\, f \in \R^V :\ & \text{there is~$F \in \R^{V \times V}$
      such that $f = \diag F$},\\
    & F_x[V_x] \in F(x, x) \thb(H_x)\quad\text{for every~$x \in V$, and}\\
    & \text{$\borderM(F)$ is positive semidefinite}\,\},
  \end{split}
\]
where, if a link~$H_x$ is empty, no constraint is imposed on the row~$F_x$.
\end{definition}

Since the links of an $r$-uniform hypergraph are $(r-1)$-uniform
hypergraphs, we have a recursive definition.  When~$r = 2$, we
have~$\thb(H_x) = \{0\}$ for every nonempty link, and so we recover the
usual definition~\eqref{eq:theta-graph} of the theta body of a graph.

The theta number can now be extended to hypergraphs: given a weight
function~$w \in \R^V$, the \defi{theta number} of $H$ for~$w$ is
\begin{equation} \label{eq:theta-number}
  \vartheta(H, w) = \max\{\, w^\tp f : f \in \thb(H)\,\}.
\end{equation}
For unit weights, we write~$\vartheta(H)$ instead of~$\vartheta(H, \one)$.

In~\S\ref{sec:theta-properties}, we will see how~$\thb(H)$ defined above is
in many ways analogous to the theta body of a graph defined
in~\eqref{eq:theta-graph}.  In particular, we will see in
Theorem~\ref{thm:theta-basic} that
\[
  \ind(H) \subseteq \thb(H) \subseteq \clique(H),
\]
and therefore~$\alpha(H, w) \leq \vartheta(H, w)$ for every weight
function~$w$.  Moreover, as shown in Theorem~\ref{thm:polytime}, it is
possible to optimize linear functions over $\thb(H)$ in polynomial time.


\subsection{The weighted fractional chromatic number}%
\label{sec:chromatic}

Let~$H = (V, E)$ be an $r$-uniform hypergraph for some~$r \geq 2$.  The
\defi{chromatic number} of~$H$, denoted by~$\chi(H)$, is the minimum number
of colors needed to color the vertices of~$H$ in such a way that no edge is
monochromatic.  In other words,~$\chi(H)$ is the minimum number of disjoint
independent sets needed to partition the vertex set of~$H$.

Given~$w \in \Rplus^V$, the \defi{weighted fractional chromatic number}
of~$H$ is
\begin{defsymb}{$\chi^*(H, w)={}$}
  minimum of~$\lambda_1 + \cdots + \lambda_k$, where~$\lambda_1$,
  \dots,~$\lambda_k \geq 0$ and there are independent sets~$I_1$,
  \dots,~$I_k$ such
  that~$\lambda_1 \chi_{I_1} + \cdots + \lambda_k \chi_{I_k} = w$.
\end{defsymb}
When~$w = \one$ is the constant-one function,~$\chi^*(H, w)$ is the
\defi{fractional chromatic number}, denoted simply by~$\chi^*(H)$.  Note
also that~$k$ is not specified, so we may consider any number of
independent sets.  In this way, if~$w = \one$ and the~$\lambda_i$ are
required to be integers, then we get the chromatic number,
so~$\chi^*(H) \leq \chi(H)$.

For the chromatic or weighted fractional chromatic number, the case~$r = 1$
is degenerate: if the hypergraph has an edge, then there is no coloring,
hence the restriction to~$r \geq 2$.

For a graph~$G = (V, E)$ and a weight function~$w \in \Rplus^V$, it is
known~\cite[Chapter~9]{GrotschelLS1988} that
$\vartheta(G, w) \leq \chi^*(\overline{G}, w)$.  (The same inequality for
the chromatic number and~${w = \one}$ was proved by
Lovász~\cite{Lovasz1979}.)  Corollary~\ref{cor:theta-chi} generalizes this
inequality to the setting of hypergraphs: if~$H = (V, E)$ is an $r$-uniform
hypergraph and~$w \in \Rplus^V$ is a weight function,
then~$\vartheta(H, w) \leq (r-1)\chi^*(\overline{H}, w)$.


\subsection{The Hoffman bound}%

The Lovász theta number is also related to a well-known spectral upper
bound for the independence number of regular graphs, originally due to
Hoffman.  If~$G$ is a $d$-regular graph on~$n$ vertices and if~$\lambda$ is
the smallest eigenvalue of its adjacency matrix, then
\[
  \alpha(G) \leq \frac{-\lambda}{d - \lambda} n;
\]
this upper bound for the independence number is known as the \defi{Hoffman
  bound}.

The Hoffman bound connects spectral graph theory with extremal
combinatorics, and as such has found many applications in combinatorics and
theoretical computer science.  Recently, it has been extended to the
high-dimensional setting of edge-weighted hypergraphs by Filmus, Golubev,
and Lifshitz~\cite{FilmusGL2021}, who also derived interesting applications
in extremal set theory.

Lovász~\cite[Theorem~9]{Lovasz1979} showed that the theta
number~$\vartheta(G)$ is always at least as good as the Hoffman bound, that
is, $\alpha(G) \leq \vartheta(G) \leq -\lambda n/(d - \lambda)$ for every
$d$-regular graph~$G$.  In~\S\ref{sec:hoffman} we will extend Lovász's
result to hypergraphs, showing that the hypergraph theta
number~$\vartheta(H)$ is also at least as good as the high-dimensional
Hoffman bound.


\subsection{The antiblocker of the theta body}%
\label{sec:antib-def}

A convex set~$K \subseteq \R^n$ is of \defi{antiblocking type}
if~$\emptyset \neq K \subseteq \Rplus^n$ and if~$x \in K$
and~$0 \leq y \leq x$ implies that~$y \in K$.  The \defi{antiblocker}
of~$K$ is
\[
  A(K) = \{\, x \in \Rplus^n : \text{$x^\tp y \leq 1$ for all~$y \in
    K$}\,\}.
\]
Note that the antiblocker of a convex set of antiblocking type is also a
convex set of antiblocking type.  If~$K$ is also assumed to be closed,
then~$A(A(K)) = K$ (see Grötschel, Lovász, and
Schrijver~\cite[p.~11]{GrotschelLS1988}).

If~$G$ is a graph, then the antiblocker of~$\thb(G)$
is~$\thb(\overline{G})$ (see Grötschel, Lovász, and
Schrijver~\cite[Chapter~9]{GrotschelLS1988}).  This fact is essential for
proving that a graph is perfect if and only if its theta body is a
polytope.

The same, however, does not hold for hypergraphs in general.
In~\S\ref{sec:antiblocker} we will describe the antiblocker of~$\thb(H)$
explicitly, and this will lead to another relaxation of~$\ind(H)$ and
corresponding bounds for the weighted independence number and the weighted
fractional chromatic number.


\subsection{Symmetry and applications}%

When a hypergraph is highly symmetric, it is possible to greatly simplify
the optimization problem giving the theta number, as we explore
in~\S\ref{sec:symmetry}.

By exploiting symmetry we are able to explicitly compute the theta number
in the following two illustrative cases.  In~\S\ref{sec:mantel} we consider
a family of hypergraphs related to Mantel's theorem in extremal graph
theory.  In this toy example, we compute the theta number of these
hypergraphs, showing that it gives a tight bound for the independence
number leading to a proof of Mantel's theorem.

In~\S\ref{sec:hamming} we consider $3$-uniform hypergraphs over the Hamming
cube whose edges are all triangles with a given side length in Hamming
distance.  We give a closed formula for the theta number, and we show
numerical results supporting our conjecture (see
Conjecture~\ref{conj:hamming}) that the density of such triangle-avoiding
sets in the Hamming cube decays exponentially fast with the dimension.


\subsection{Notation}

For an integer~$n\geq 1$ we write~$[n] = \{1, \ldots, n\}$.  For a set~$V$
and~$S \subseteq V$ we denote by~$\chi_S\colon V \to \R$ the
\defi{characteristic function} of~$S$, which is defined by~$\chi_S(x) = 1$
if~$x \in S$ and~$\chi_S(x) = 0$ otherwise.  If~$f\colon V \to \R$ is a
function and~$S \subseteq V$, then~$f(S) = \sum_{x \in S} f(x)$.  The
collection of all $r$-subsets of~$V$ is denoted by~$\binom{V}{r}$.

If $H = (V, E)$ is an $r$-uniform hypergraph, we denote by $\overline{H}$
the \defi{complement} of~$H$, which is the hypergraph with vertex set $V$
and edge set $\binom{V}{r} \setminus E$.

We denote by~$\diag A$ the vector giving the diagonal of a square
matrix~$A$.  The \defi{trace inner product} between symmetric matrices~$A$,
$B \in \R^{n \times n}$
is~$\langle A, B \rangle = \trace AB = \sum_{i,j=1}^n A_{ij} B_{ij}$.
Positive semidefinite matrices are always symmetric.  For a symmetric
matrix~$A$ with diagonal~$a$, we write
\[
  \borderM(A) =
  \begin{pmatrix}
    1&a^\tp\\a&A
  \end{pmatrix}.
\]


\section{Properties of the theta body}%
\label{sec:theta-properties}

Given an $r$-uniform hypergraph~$H = (V, E)$ for~$r \geq 2$, it is useful
to consider the lifted version of the theta body as given in
Definition~\ref{def:theta-body}, namely
\[
  \begin{split}
    \lthb(H) = \{\, F \in \R^{V \times V} :\
    &F_x[V_x] \in F(x, x) \thb(H_x)\text{ for every~$x \in V$ and}\\
    &\text{$\borderM(F)$ is positive semidefinite}\,\}.
  \end{split}
\]
Note that~$\thb(H)$ is the projection of~$\lthb(H)$ onto the subspace of
diagonal matrices, being therefore a projected spectrahedron.

\begin{theorem}%
  \label{thm:theta-basic}
  If~$H$ is an $r$-uniform hypergraph, then~$\thb(H)$ is compact, convex,
  and satisfies
  \begin{equation}
    \label{eq:theta-inclusion}
    \ind(H) \subseteq \thb(H) \subseteq \clique(H).
  \end{equation}
\end{theorem}

\begin{proof}
The proof proceeds by induction on~$r$.  The base case is~$r = 1$, for
which the statement is easily seen to hold.

Assume~$r \geq 2$.  By the induction hypothesis, the statement of the
theorem holds for the theta body of every link.  This implies
that~$\lthb(H)$ is convex, and hence~$\thb(H)$ is convex.

Let us show that $\lthb(H)$ is compact and, since $\thb(H)$ is a projection
of $\lthb(H)$, it will follow that $\thb(H)$ is compact.

Let~$(F_k)_{k \geq 1}$ be a sequence of points in~$\lthb(H)$ that converges
to~$F$.  Immediately we have that~$\borderM(F)$ is positive semidefinite.
Now fix~$x \in V$ and let~$a^\tp f \leq \beta$ be any valid inequality
for~$\thb(H_x)$.  Then
\[
  a^\tp F_x[V_x] = \lim_{k\to\infty} a^\tp (F_k)_x[V_x] \leq
  \lim_{k\to\infty} F_k(x, x) \beta = F(x, x)\beta,
\]
and we see that~$F_x[V_x] \in F(x, x) \thb(H_x)$, proving that~$\lthb(H)$
is closed.

To see that~$\lthb(H)$ is bounded, note that for every~$x \in V$
the~$2\times 2$ submatrix
\[
  \begin{pmatrix}
    1&f(x)\\
    f(x)&f(x)
  \end{pmatrix}
\]
of~$\borderM(F)$ is positive semidefinite (where~$f = \diag F$),
hence~$f(x) - f(x)^2 \geq 0$ and so~$|F(x, x)| = |f(x)| \leq 1$ for
all~$x \in V$.  This implies that $\trace F \leq |V|$ and, since~$F$ is
positive semidefinite, the Frobenius norm\footnote{The \defi{Frobenius
    norm} of a symmetric matrix~$A \in \R^{n \times n}$ is the Euclidean
  norm of~$A$ considered as an $n^2$-dimensional vector.  If~$A$ has
  eigenvalues~$\lambda_1$, \dots,~$\lambda_n$, then the square of the
  Frobenius norm is~$\lambda_1^2 + \cdots + \lambda_n^2$, hence the
  Frobenius norm is at most~$\trace A = \lambda_1 + \cdots + \lambda_n$
  when all eigenvalues are nonnegative, that is, when~$A$ is positive
  semidefinite.} of~$F$ is at most~$|V|$.  This finishes the proof
that~$\lthb(H)$ is compact.

It remains to show that~\eqref{eq:theta-inclusion} holds.  For the first
inclusion, let~$I \subseteq V$ be an independent set.  For every~$x \in V$,
if~$x \in I$, then~$I \cap V_x$ is an independent set of the link~$H_x$, so
by the induction hypothesis~$\chi_I[V_x] \in \thb(H_x)$.  It follows
that~$\chi_I \chi_I^\tp \in \lthb(H)$, and so $\ind(H) \subseteq \thb(H)$.

For the second inclusion in~\eqref{eq:theta-inclusion}, note first that
$\thb(H) \subseteq [0, 1]^V$.  Let~$C \subseteq V$ be a clique and
let~$F \in \lthb(H)$; write~$f = \diag F$.  If~$|C| \leq r - 1$, then
since~$f \leq \one$ we have~$f(C) \leq |C| \leq r - 1$ and we are done.

So assume~$|C| \geq r$.  Since~$\borderM(F)$ is positive semidefinite we
have
\begin{equation}
  \label{eq:clique-partial}
  0 \leq (r-1; -\chi_C)^\tp \borderM(F) (r-1; -\chi_C)
  = (r-1)^2 - 2(r-1)\chi_C^\tp f + \chi_C^\tp F \chi_C.
\end{equation}
Since~$|C| \geq r$, every $r$-element subset of~$C$ is an edge of~$H$, and
so for every~$x \in C$ we know that~$C \setminus \{x\} \subseteq V_x$ is a
clique of the link~$H_x$.  By the induction hypothesis we know
that~$\thb(H_x) \subseteq \clique(H_x)$, hence
\[
  \begin{split}
    \chi_C^\tp F \chi_C = \sum_{x, y \in C} F(x, y)
    &= \sum_{x \in C} \biggl(F(x, x) + \sum_{y \in C \setminus \{x\}} F(x, y)\biggr)\\
    &\leq \sum_{x \in C} (F(x, x) + F(x, x) (r - 2)) = (r-1)
    f(C).
  \end{split}
\]
Together with~\eqref{eq:clique-partial} we get
$0 \leq (r-1)^2 - (r-1)f(C)$, whence~$f(C) \leq r-1$ as wished.
\end{proof}

As a corollary we get that~$\thb(H)$ is a formulation of~$\ind(H)$, that
is, the integer hull of the theta body is the independent-set polytope.

\begin{corollary}
  If~$H$ is an $r$-uniform hypergraph and if~$f \in \thb(H)$ is an integral
  vector, then~$f$ is the characteristic function of an independent set
  of~$H$.
\end{corollary}

\begin{proof}
As~$\thb(H) \subseteq \clique(H) \subseteq [0, 1]^V$, we know that~$f$
is a~0--1 vector that satisfies all clique inequalities, and the conclusion
follows.
\end{proof}

Since~$\ind(H) \subseteq \thb(H)$, it follows immediately from the
definition~\eqref{eq:theta-number} that $\alpha(H, w) \leq \vartheta(H, w)$
for every~$w \in \R^V$.  What about a lower bound for the chromatic number?

For a graph~$G = (V, E)$ and~$w \in \Rplus^V$ we also
have~$\vartheta(G, w) \leq \chi^*(\overline{G}, w)$.  (Recall the
definition of $\chi^*$ from \S\ref{sec:chromatic}.)  A simple example shows
that the same cannot be true for hypergraphs in general.  Indeed,
fix~$r \geq 3$ and let~$H$ be the complete $r$-uniform hypergraph on~$r$
vertices (that is,~$H$ has exactly one edge containing all of its
vertices).  The complement of~$H$ is the empty hypergraph.
Then~$\vartheta(H) = r-1$, whereas~$\chi^*(\overline{H}) = 1$, and the
inequality fails to hold.  It can, however, be extended, and this simple
example shows that the extension is tight.

\begin{corollary}%
\label{cor:theta-chi}%
If~$H$ is an $r$-uniform hypergraph and~$w \in \Rplus^V$,
then~$\alpha(H, w) \leq \vartheta(H, w)$.  If moreover~$r \geq 2$,
then~$\vartheta(H, w) \leq (r-1) \chi^*(\overline{H}, w)$.
\end{corollary}

\begin{proof}
The first statement follows immediately from~$\ind(H) \subseteq \thb(H)$.

The second statement follows from~$\thb(H) \subseteq \clique(H)$.  Indeed,
let~$\lambda_1$, \dots,~$\lambda_k$ be nonnegative numbers and~$C_1$,
\dots,~$C_k$ be independent sets of~$\overline{H}$ such
that~$\lambda_1 \chi_{C_1} + \cdots + \lambda_k \chi_{C_k} = w$.
If~$f \in \thb(H)$, then~$f$ satisfies all clique inequalities, so since
each~$C_i$ is a clique of~$H$ we have~$\chi_{C_i}^\tp f = f(C_i) \leq r-1$.
Hence
\[
  w^\tp f = (\lambda_1 \chi_{C_1} + \cdots + \lambda_k \chi_{C_k})^\tp f
  \leq (r-1) (\lambda_1 + \cdots + \lambda_k),
\]
and we are done.
\end{proof}

Just like the theta body of a graph, the theta body of a hypergraph can be
shown to be of antiblocking type (see~\S\ref{sec:antib-def} for
background), and this leads to an inequality description of the theta body
in terms of the theta number.

\begin{theorem}%
\label{thm:theta-antiblocking}%
If~$H = (V, E)$ is an $r$-uniform hypergraph, then~$\thb(H)$ is of
antiblocking type and $\thb(H) = \{\, f \in \Rplus^V :
\text{$w^\tp f \leq \vartheta(H, w)$ for all~$w \in \Rplus^V$}\, \}$.
\end{theorem}

\begin{proof}
We proceed by induction on~$r$.  The statement is immediate for the base
case~$r = 1$, so assume~$r \geq 2$.  We claim: if~$w \in \R^V$
and~$w_+(x) = \max\{0, w(x)\}$ for all~$x \in V$,
then~$\vartheta(H, w) = \vartheta(H, w_+)$.

Since $\thb(H) \subseteq \Rplus^V$, it is clear that
$\vartheta(H, w) \leq \vartheta(H, w_+)$ for every~$w \in \R^V$; let us now
prove the reverse inequality.  Let~$F \in \lthb(H)$ be a matrix such
that~$w_+^\tp f = \vartheta(H, w_+)$, where~$f = \diag F$.
Let~$S = \{\, x \in V : w(x) \geq 0\,\}$ and denote by~$\bar{F}$ the
Hadamard (entrywise) product of~$F$ and~$\chi_S \chi_S^\tp$;
write~$\bar{f} = \diag\bar{F}$.

Note that~$\borderM(\bar{F})$ is the Hadamard product of~$\borderM(F)$
and~$(1; \chi_S)(1; \chi_S)^\tp$, hence $\borderM(\bar{F})$ is positive
semidefinite.  For every~$x \in V$ we have
$0 \leq \bar{F}_x[V_x] \leq F_x[V_x]$, and so the induction hypothesis
implies that~$\bar{F}_x[V_x] \in \bar{F}(x, x) \thb(H_x)$.
Hence~$\bar{F} \in \lthb(H)$, and
$\vartheta(H, w) \geq w^\tp \bar{f} = w_+^\tp f = \vartheta(H, w_+)$,
proving the claim.

The inequality description follows immediately.  Every~$f \in \thb(H)$ is
nonnegative and satisfies~$w^\tp f \leq \vartheta(H, w)$ for
all~$w \in \Rplus^V$.  For the reverse inclusion note that, since~$\thb(H)$
is closed and convex,
\[
  \thb(H) = \{\, f \in \R^V : \text{$w^\tp f \leq \vartheta(H, w)$ for
    all~$w \in \R^V$}\, \}.
\]
So let~$f\geq 0$ be such that~$w^\tp f \leq \vartheta(H, w)$ for
all~$w \in \Rplus^V$.  For~$w \in \R^V$, let~$w_+$ be defined as above,
so~$w_+ \geq 0$.  Then, for every~$w \in \R^V$ we have
\[
  w^\tp f \leq w_+^\tp f \leq \vartheta(H, w_+) = \vartheta(H, w),
\]
and we see that~$f \in \thb(H)$.

To finish, let us show that the theta body is of antiblocking type.
If~$f \in \thb(H)$ and~$0\leq g \leq f$, then for every~$w \in \Rplus^V$ we
have $w^\tp g \leq w^\tp f \leq \vartheta(H, w)$, and so~$g \in \thb(H)$.
\end{proof}

Finally, for every fixed~$r \geq 1$ it is possible to optimize
over~$\thb(H)$ in polynomial time.  More precisely, in the language of
Grötschel, Lovász, and Schrijver~\cite[Chapter~4]{GrotschelLS1988}, we
have:

\begin{theorem}%
\label{thm:polytime}%
If~$r \geq 1$ is fixed, then the weak optimization problem over~$\thb(H)$
can be solved in polynomial time for every $r$-uniform hypergraph~$H$.
\end{theorem}

\begin{proof}
The result is trivial for~$r = 1$.  For graphs, that is,~$r = 2$, the
statement was proven by Grötschel, Lovász, and
Schrijver~\cite[Theorem~9.3.30]{GrotschelLS1988}, and here it is easier to
take~$r = 2$ as our base case, as will become clear soon.  So we assume
that~$r \geq 3$ and that the weak optimization problem can be solved in
polynomial time for $(r-1)$-uniform hypergraphs; we want to show how to
solve the weak optimization problem in polynomial time for $r$-uniform
hypergraphs.
 
Let~$H = (V, E)$ be an $r$-uniform hypergraph.  If we show that we can solve the
weak optimization problem over the convex set~$\lthb(H)$, then we are done.  It
suffices~\cite[Chapter~4]{GrotschelLS1988} to show that $\lthb(H)$ has the
required inscribed and circumscribed balls of appropriate size, and that the
weak membership problem for $\lthb(H)$ can be solved in polynomial time.

It can be easily checked that all entries of a matrix in $\lthb(H)$ are
bounded in absolute value by $1$, and so the origin-centered ball of radius
$|V|$ circumscribes the theta body.  To find an inscribed ball, note that
the full-dimensional convex set
\[
  \conv\{\, \chi_I \chi_I^\tp \in \R^{V \times V} : \text{$I \subseteq V$,
    $|I| \leq 2$}\,\}
\]
is a subset of $\lthb(H)$, so it contains a ball which is also contained in
$\lthb(H)$.  (This assertion fails when~$H$ is a graph, which is why we
take~$r = 2$ as the base to simplify the proof.)

Now, given a symmetric matrix~$F \in \R^{V \times V}$, to test whether
$F \in \lthb(H)$ we first test whether~$\borderM(F)$ is positive
semidefinite using (for instance) Cholesky decomposition.  By induction,
the weak optimization problem for each link can be solved in polynomial
time, hence so can the weak membership problem for each link.  We then
finish by calling the weak membership oracle for each link.
\end{proof}


\section{Relationship to the Hoffman bound}%
\label{sec:hoffman}

Let~$G$ be a $d$-regular graph on~$n$ vertices and let~$\lambda$ be the
smallest eigenvalue of its adjacency matrix.  Hoffman showed that
\[
  \alpha(G) \leq \frac{-\lambda}{d-\lambda} n;
\]
the right-hand side above became know as the \textit{Hoffman bound}.
Hoffman never published this particular result, though he did publish a
similar lower bound~\cite{Hoffman1970} for the chromatic number which also
came to be known as the \textit{Hoffman bound}; see
Haemers~\cite{Haemers2021} for a historical account.
Lovász~\cite{Lovasz1979} showed that the theta number is always at least as
good as the Hoffman bound and that, when the graph is edge transitive, both
bounds coincide.  The Hoffman bound has also been extended to certain
infinite graphs, and its relation to extensions of the theta number has
been studied~\cite{BachocDOV2014}.

Filmus, Golubev, and Lifshitz~\cite{FilmusGL2021} extended the Hoffman bound to
edge-weighted hypergraphs and described several applications to extremal
combinatorics.  Our goal in this section is to show that our extension of the
theta number to hypergraphs is always at least as good as the extended Hoffman
bound.  We begin with the extension of Filmus, Golubev, and Lifshitz.

To simplify the presentation and to be consistent with the setup used so
far, we restrict ourselves to weighted hypergraphs without loops.  A
\defi{weighted $r$-uniform hypergraph} is a pair~$X = (V, \mu)$ where~$V$
is a finite set, called the \defi{vertex set} of the hypergraph, and~$\mu$
is a probability measure on~$\binom{V}{r}$.  The \defi{underlying
  hypergraph} of~$X$ is the $r$-uniform hypergraph on~$V$ whose edge set is
the support of~$\mu$.

Let~$X = (V, \mu)$ be a weighted $r$-uniform hypergraph and
let~$H = (V, E)$ be its underlying hypergraph.  For~$i = 1$, \dots,~$r-1$,
the measure~$\mu$ induces a probability measure~$\mu^{(i)}$
on~$\binom{V}{i}$ by the following experiment: we first choose an edge~$e$
of~$X$ according to~$\mu$ and then we choose an $i$-subset of~$e$ uniformly
at random.  Concretely, for~$\sigma \in \binom{V}{i}$ we have
\begin{equation}%
  \label{eq:mu-i}
  \mu^{(i)}(\sigma) = \binom{r}{i}^{-1} \mu(\{\, e \in E : \sigma \subseteq
  e\,\}).
\end{equation}
Note that~$\mu^{(1)}$ can be seen as a weight function on~$V$.  We define
the \defi{independence number} of~$X$
as~$\alpha(X) = \alpha(H, \mu^{(1)})$.

Let~$X^{(i)} \subseteq \binom{V}{i}$ be the support of~$\mu^{(i)}$.  We may
assume, without loss of generality, that~$X^{(1)} = V$, since vertices not
in the support of~$\mu^{(1)}$ are isolated and do not contribute to the
independence number.

The \defi{link} of~$\sigma \in X^{(i)}$ is the weighted $(r-i)$-uniform
hypergraph~$X_\sigma = (V, \mu_\sigma)$, where~$\mu_\sigma$ is the
probability measure on~$\binom{V}{r-i}$ defined by the following
experiment: sample a random edge~$e \in \binom{V}{r}$ according to~$\mu$
conditioned on~$\sigma \subseteq e$ and output~$e \setminus \sigma$.  We
also say that~$X_\sigma$ is an \defi{$i$-link} of~$X$.  Concretely,
for~$e \in \binom{V \setminus \sigma}{r-i}$ we have
\begin{equation}%
  \label{eq:mu-sigma}
  \mu_\sigma(e) = \frac{\mu(e \cup \sigma)}{\mu(\{\, e' \in E : \sigma
    \subseteq e'\,\})}.
\end{equation}
For a vertex~$x \in V$ we write~$X_x$ instead of~$X_{\{x\}}$ for the link
of~$x$.  Note that the underlying hypergraph of~$X_x$, minus its isolated
vertices, is exactly~$H_x$, the link of~$x$ in the underlying hypergraph
of~$X$, which we have used so far.

Equip~$\R^V$ with the inner product
\[
  \lhof f, g\rhof = \sum_{x \in V} f(x) g(x) \mu^{(1)}(x)
\]
for~$f$, $g \in \R^V$.  Since~$V$ is the support of~$\mu^{(1)}$, this inner
product is nondegenerate.  The \defi{normalized adjacency operator} of~$X$
is the operator~$T_X$ on~$\R^V$ such that
\[
  (T_X f)(x) = \sum_{y \in V} f(y) \mu_x^{(1)}(y)
\]
for all~$f \in \R^V$.  Here,~$\mu_x^{(1)} = (\mu_x)^{(1)}$ is the measure
on~$V$ induced by the measure~$\mu_x$ defining the link of~$x$.
Combine~\eqref{eq:mu-i} and~\eqref{eq:mu-sigma} to get
\begin{equation}%
  \label{eq:mu-x-1}
  \mu_x^{(1)}(y) = \frac{1}{r-1} \frac{\mu(\{\, e \in E : x, y \in
    e\,\})}{\mu(\{\, e \in E : x \in e\, \})}.
\end{equation}
Now use~\eqref{eq:mu-i} and~\eqref{eq:mu-x-1} to see that
\begin{equation}%
  \label{eq:mu-x-1-expr}
  \mu_x^{(1)}(y) = \frac{\mu^{(2)}(\{x, y\})}{2\mu^{(1)}(x)}
\end{equation}
for every~$x \in V$ and~$y \in V_x$.  Hence
\begin{equation}%
  \label{eq:TX}
  \lhof T_X f, g\rhof = \sum_{x \in V} \sum_{y \in V} f(y)
  \mu_x^{(1)}(y) g(x) \mu^{(1)}(x) = \sum_{\{x, y\} \in \binom{V}{2}} f(x) g(y)
  \mu^{(2)}(\{x, y\})
\end{equation}
for~$f$, $g \in \R^V$.  It follows at once that~$T_X$ is self-adjoint and
thus has real eigenvalues.

Note that~$T_X\one = \one$, hence the constant one vector is an eigenvector
of~$T_X$ with associated eigenvalue~$1$.  Moreover, the largest eigenvalue
of~$T_X$ is~$1$.  Indeed, recall that if~$A \in \R^{n \times n}$ is a
matrix and if~$\lambda$ is an eigenvalue of~$A$,
then~$|\lambda| \leq \|A\|_\infty = \max_{i \in [n]} \sum_{j=1}^n
|A_{ij}|$.  Since~$\|T_X\|_\infty = 1$ by construction, it follows that~$1$
is the largest eigenvalue of~$T_X$.

Let~$\lambda(X)$ be the smallest eigenvalue of~$T_X$, which is negative
since~$\trace T_X = 0$ as is clear from~\eqref{eq:TX}.  For~$i = 1$,
\dots,~$r-2$, let~$\lambda_i(X)$ be the minimum possible eigenvalue of the
normalized adjacency operator of any $i$-link of~$X$, that is,
\[
  \lambda_i(X) = \min_{\sigma \in X^{(i)}} \lambda(X_\sigma),
\]
and set $\lambda_0(X) = \lambda(X)$.

With this notation, the \textit{Hoffman bound} of~$X$ introduced by Filmus,
Golubev, and Lifshitz~\cite{FilmusGL2021} is
\[
  \hoff(X) = 1 - \frac{1}{(1-\lambda_0(X))(1 - \lambda_1(X)) \cdots (1 -
    \lambda_{r-2}(X))}.
\]

Say~$G = (V, E)$ is a $d$-regular graph and introduce on its edges the
uniform probability measure, obtaining a weighted $2$-uniform
hypergraph~$X = (V, \mu)$.  In this case,~$\mu^{(1)}$ is the uniform
probability measure on~$V$ and the normalized adjacency operator~$T_X$ is
simply the adjacency matrix of~$G$ divided by~$d$.  If~$\lambda$ is the
smallest eigenvalue of the adjacency matrix of~$G$, then~$\lambda(X) =
\lambda / d$ and the high-dimensional Hoffman bound reads
\[
  \hoff(X) = 1 - \frac{1}{1-\lambda_0(X)} = \frac{-\lambda}{d-\lambda},
\]
which is, up to normalization, the Hoffman bound for~$\alpha(G)$.

Filmus, Golubev, and Lifshitz showed that~$\alpha(X) \leq \hoff(X)$ and
that this bound does not change when one takes tensor powers of the
hypergraph, a fact that has implications for some problems in extremal
combinatorics.  The next theorem relates the hypergraph theta number to the
high-dimensional Hoffman bound.

\begin{theorem}%
  \label{thm:hoff-theta}
  If~$X = (V, \mu)$ is a weighted $r$-uniform hypergraph for
  some~$r \geq 2$ and if~$H$ is its underlying hypergraph,
  then~$\alpha(X) \leq \vartheta(H, \mu^{(1)}) \leq \hoff(X)$.
\end{theorem}

A few remarks before the proof.  The theta number is a bound for the
weighted independence number, where the weights are placed on the vertices.
The Hoffman bound on the other hand is defined for an edge-weighted
hypergraph, and since edge weights naturally induce vertex weights, it is
possible to compare it to the theta number.  However, not every weight
function on vertices can be derived from a weight function on edges, so in
this sense the theta number applies in more general circumstances.

Moreover, even when a vertex-weight function~$w\colon V \to \Rplus$ can be
derived from an edge-weight function, it is not clear how to efficiently
find an edge-weight function that gives~$w$ and for which the Hoffman bound
gives a good upper bound for~$\alpha(H, w)$.  A natural idea is to compute
the optimal Hoffman bound, that is, to find the edge weights inducing~$w$
for which the corresponding Hoffman bound is the smallest possible.  This
was proposed by Filmus, Golubev, and Lifshitz~\cite[\S4.3]{FilmusGL2021},
but for~$r \geq 3$ the resulting optimization problem has a nonconvex
objective function, and it is not clear how to solve it efficiently.  In
contrast, one can always efficiently compute the theta number of a
hypergraph (see Theorem~\ref{thm:polytime}), and
Theorem~\ref{thm:hoff-theta} says that the bound so obtained will always be
at least as good as the optimal Hoffman bound.

Finally, an important property of the extension of the Hoffman bound is
that it is invariant under the tensor power operation, while it is unclear
whether the hypergraph theta number behaves nicely under natural hypergraph
products.

\begin{proof}[Proof of Theorem~\ref{thm:hoff-theta}]
By definition we have~$\alpha(X) = \alpha(H, \mu^{(1)})$ which, by
Corollary~\ref{cor:theta-chi}, is at most~$\vartheta(H, \mu^{(1)})$.

The proof of the inequality~$\vartheta(H, \mu^{(1)}) \leq \hoff(X)$
proceeds by induction on~$r$.  The base is~$r = 2$, in which case the
statement was shown by Lovász~\cite[Theorem~9]{Lovasz1979}.  (Note that the
Hoffman bound is not defined for~$r = 1$, which is why we take~$r = 2$ as
the base.)

So assume~$r \geq 3$.  Let~$f \in \thb(H)$ be such
that~$(\mu^{(1)})^\tp f = \vartheta(H, \mu^{(1)})$ and let~$F \in \lthb(H)$
be a matrix such that~$f = \diag F$.  Since~$\borderM(F)$ is positive
semidefinite, by taking the Schur complement we see that~$F - f f^\tp$ is
also positive semidefinite, and so
\[
  \sum_{x, y \in V} F(x, y) \mu^{(1)}(x) \mu^{(1)}(y) \geq ((\mu^{(1)})^\tp
  f)^2 = \vartheta(H, \mu^{(1)})^2.
\]
To finish the proof it then suffices to show that
\begin{equation}%
  \label{eq:hoff-goal}
  \sum_{x, y \in V} F(x, y) \mu^{(1)}(x) \mu^{(1)}(y) \leq \vartheta(H, \mu^{(1)})
  \hoff(X).
\end{equation}

Since~$F \in \lthb(H)$, we have $F_x[V_x] \in F(x, x) \thb(H_x)$ for
every~$x \in V$, and so
\[
  \sum_{y \in V_x} F(x, y) \mu_x^{(1)}(y) \leq F(x, x) \vartheta(H_x,
  \mu_x^{(1)}).
\]
By induction, $\vartheta(H_x, \mu_x^{(1)}) \leq \hoff(X_x)$, hence
\[
  \sum_{x \in V} \mu^{(1)}(x) \sum_{y \in V_x} F(x, y) \mu_x^{(1)}(y) \leq
  \sum_{x \in V} \mu^{(1)}(x) F(x, x) \vartheta(H_x, \mu_x^{(1)}) \leq
  \vartheta(H, \mu^{(1)}) M,
\]
where~$M = \max_{x \in V} \hoff(X_x)$.  Use~\eqref{eq:mu-x-1-expr} on the
left-hand side above to get
\begin{equation}%
  \label{eq:double-sum-M}
  \sum_{\{x,y\} \in \binom{V}{2}} F(x, y) \mu^{(2)}(\{x, y\}) \leq \vartheta(H,
  \mu^{(1)}) M,
\end{equation}
which already looks much closer to~\eqref{eq:hoff-goal}.

We work henceforth on the space~$\R^V$ equipped with the nondegenerate inner
product~$\lhof \cdot,\cdot\rhof$ defined above.  Since~$F$ is positive
semidefinite, let~$g_1$, \dots,~$g_n$ be an orthonormal basis of
eigenvectors of~$F$, with associated nonnegative eigenvalues~$\lambda_1$,
\dots,~$\lambda_n$.  We then
have~$F(x, y) = \sum_{i=1}^n \lambda_i g_i(x) g_i(y)$ and
\begin{gather}%
  \label{eq:hoff-trace}
  \sum_{i=1}^n \lambda_i = \sum_{i=1}^n \lambda_i \lhof g_i, g_i \rhof =
  \sum_{x\in V} F(x, x) \mu^{(1)}(x) = \vartheta(H, \mu^{(1)}),\\
  \label{eq:hoff-double}
  \sum_{i=1}^n \lambda_i \lhof g_i, \one\rhof^2 = \sum_{x, y \in V} F(x, y)
  \mu^{(1)}(x) \mu^{(1)}(y),
\end{gather}
and, using~\eqref{eq:TX},
\begin{equation}%
  \label{eq:double-sum}
  \begin{split}
    \sum_{i=1}^n \lambda_i \lhof T_X g_i, g_i\rhof &= \sum_{i=1}^n
    \lambda_i \sum_{\{x,y\} \in \binom{V}{2}} g_i(x) g_i(y)
    \mu^{(2)}(\{x,y\})\\
    &= \sum_{\{x,y\} \in \binom{V}{2}} F(x, y) \mu^{(2)}(\{x,y\}).
  \end{split}
\end{equation}

Let~$\one = v_1$, $v_2$, \dots,~$v_n$ be an orthonormal basis of
eigenvectors of~$T_X$ with associated
eigenvalues~$1 = \alpha_1 \geq \alpha_2 \geq \cdots \geq \alpha_n =
\lambda(X)$.  For every~$i$ we have
\[
  1 = \lhof g_i, g_i \rhof = \sum_{j=1}^n \lhof g_i, v_j\rhof^2,
\]
whence~$\sum_{j=2}^n \lhof g_i, v_j\rhof^2 = 1 - \lhof g_i, \one\rhof^2$.
It follows that
\[
  \begin{split}
    \lhof T_X g_i, g_i\rhof &= \sum_{j=1}^n \alpha_j \lhof g_i, v_j\rhof^2\\
    &\geq \lhof g_i, \one\rhof^2 + \sum_{j=2}^n \alpha_n \lhof g_i,
    v_j\rhof^2\\
    &= (1 - \alpha_n) \lhof g_i, \one\rhof^2 + \alpha_n.
  \end{split}
\]

Combine this with~\eqref{eq:double-sum} and~\eqref{eq:double-sum-M} to get
\[
  \sum_{i=1}^n \lambda_i ((1 - \alpha_n)\lhof g_i, \one\rhof^2 + \alpha_n)
  \leq \sum_{\{x,y\} \in \binom{V}{2}} F(x, y) \mu^{(2)}(\{x,y\})\leq
  \vartheta(H, \mu^{(1)}) M.
\]
By~\eqref{eq:hoff-trace} this implies that
\[
  (1-\alpha_n)\sum_{i=1}^n \lambda_i \lhof g_i, \one\rhof^2 \leq
  \vartheta(H, \mu^{(1)}) (M - \alpha_n).
\]
Since~$\alpha_n < 0$ and hence~$1 - \alpha_n > 0$,
using~\eqref{eq:hoff-double} we finally get
\begin{equation}%
  \label{eq:hoff-end}
  \sum_{x, y \in V} F(x, y) \mu^{(1)}(x) \mu^{(1)}(y) = \sum_{i=1}^n
  \lambda_i \lhof g_i, \one\rhof^2
  \leq \vartheta(H, \mu^{(1)}) \frac{M - \alpha_n}{1 - \alpha_n}.
\end{equation}

We are now essentially done.  Indeed,~$\lambda_i(X_x) \geq
\lambda_{i+1}(X)$ for all~$x \in V$ and~$i = 0$, \dots,~$r-2$, hence
\[
  1 - \frac{1}{(1 -\lambda_0(X_x)) \cdots (1 - \lambda_{r-3}(X_x))}
  \leq
  1 - \frac{1}{(1 - \lambda_1(X)) \cdots (1 - \lambda_{r-2}(X))}.
\]
The left-hand side above is precisely~$\hoff(X_x)$ and the right-hand side
is equal to~$\lambda_0(X) + (1 - \lambda_0(X)) \hoff(X)$.  We conclude that
\[
  \frac{M - \lambda_0(X)}{1 - \lambda_0(X)} = \max_{x \in V}
  \frac{\hoff(X_x) - \lambda_0(X)}{1 - \lambda_0(X)} \leq \hoff(X).
\]
Since~$\alpha_n = \lambda_0(X)$ by definition, this combined
with~\eqref{eq:hoff-end} gives~\eqref{eq:hoff-goal}, as wished.
\end{proof}

We mentioned above that the Hoffman bound coincides with the theta number
when the graph is edge transitive.  More generally, if a weighted
hypergraph and all its lower-order links are vertex transitive, then the
Hoffman bound coincides with the theta number.  The proof of this assertion
is an adaptation of the proof of Theorem~\ref{thm:hoff-theta}: use the
results of~\S\ref{sec:symmetry} to take an invariant
matrix~$F \in \lthb(H)$ and check that since the hypergraph and all its
lower-order links are vertex transitive, every inequality in the proof is
tight.


\section{The antiblocker}%
\label{sec:antiblocker}

Recall the definition of antiblocker from~\S\ref{sec:antib-def}.  If~$G$ is
a graph, then the antiblocker of~$\thb(G)$ is~$\thb(\overline{G})$ (see
Grötschel, Lovász, and Schrijver~\cite[Chapter~9]{GrotschelLS1988}).  The
same does not hold for hypergraphs in general.  Consider, for instance, the
hypergraph~$H = ([r], \{[r]\})$ and notice
that~$f = \chi_{[r-1]} \in \ind(H)$
and~$g = \chi_{[r]} \in \ind(\overline{H})$.  So~$f \in \thb(H)$
and~$g \in \thb(\overline{H})$, but~$f^\tp g = r-1$.  Hence,
for~$r \geq 3$,~$\thb(\overline{H})$ is not the antiblocker of~$\thb(H)$.

It seems from this simple example that we are off by a factor of~$r-1$, so
is~$A(\thb(H)) = (r-1)^{-1} \thb(\overline{H})$?  The answer is again no,
and the smallest example is the hypergraph on~$\{1, \ldots,
5\}$ with edges~$\{1, 2, 5\}$, $\{1, 3, 4\}$, and~$\{2, 3, 4\}$.

To describe the antiblocker of the theta body, we start by defining an
alternative theta number inspired by the dual of the theta number for
graphs.  For a number~$\lambda$ and a symmetric matrix~$A$ with
diagonal~$a$, write
\[
  \borderM(\lambda, A) =
  \begin{pmatrix}
    \lambda&a^\tp\\a&A
  \end{pmatrix}.
\]
Given an $r$-uniform hypergraph~$H$ for~$r \geq 2$ and a weight
function~$w \in \Rplus^V$, denote by~$\btheta(H, w)$ both the
semidefinite program below and its optimal value:
\begin{equation}%
  \label{opt:thetabar}%
  \begin{optprob}
    \min&\lambda\\
    &\diag Z = w,\\
    &Z_x[\overline{V}_x] \in Z(x, x) \thb(\overline{H}_x)&\text{for all~$x
      \in V$},\\
    &\multicolumn{2}{l}{\text{$Z \in \R^{V \times V}$ and
        $\borderM(\lambda, Z)$ is positive semidefinite,}}
  \end{optprob}
\end{equation}
where~$\overline{V}_x$ is the vertex set of the link of~$x$
in~$\overline{H}$.

Now define
\[
  \bthb(H) = \{\, g \in \Rplus^V : \text{$w^\tp g \leq
    \btheta(H, w)$ for all~$w \in \Rplus^V$}\,\}.
\]
This is a nonempty, closed, and convex set of antiblocking type contained
in~$[0, 1]^V$.  Indeed, to see that~$\bthb(H) \subseteq [0, 1]^V$,
fix~$x \in V$ and let~$w = \chi_{\{x\}}$.  Then~$\lambda = 1$
and~$Z = \chi_{\{x\}} \chi_{\{x\}}^\tp$ is a feasible solution
of~$\btheta(H, w)$, whence~$\btheta(H, w) \leq 1$, implying
that~$g(x) = w^\tp g \leq 1$ for every~$g \in \bthb(H)$, as we wanted.

\begin{theorem}%
\label{thm:antiblocker}%
If~$H = (V, E)$ is an $r$-uniform hypergraph for~$r \geq 2$, then:
\begin{enumerate}
\item $\btheta(H, w) = \max\{\, w^\tp g : g \in \bthb(H)\,\}$
  for every~$w \in \Rplus^V$;

\item $\vartheta(H, l) \btheta(\overline{H}, w) \geq l^\tp w$
  for every~$l$, $w \in \Rplus^V$;

\item $A(\thb(H)) = \bthb(\overline{H})$.
\end{enumerate}
\end{theorem}

If~$G$ is a graph, then~$A(\thb(G)) = \thb(\overline{G})$, and
hence~$\thb(G) = \bthb(G)$.  For $r$-uniform hypergraphs with~$r \geq 3$,
this is no longer always the case.

\begin{proof}
The proof of~(1) will require the following facts:
\begin{enumerate}
\item[(i)] if~$w \in \Rplus^V$ and~$\alpha \geq 0$,
  then~$\btheta(H, \alpha w) = \alpha \btheta(H, w)$;
    
\item[(ii)] if~$w$, $w' \in \Rplus^V$, then
  $\btheta(H, w + w') \leq \btheta(H, w) + \btheta(H, w')$;

\item[(iii)] if~$w$, $w' \in \Rplus^V$ and~$w' \leq w$,
  then~$\btheta(H, w') \leq \btheta(H, w)$.
\end{enumerate}

To show~(i), note that if~$(\lambda, Z)$ is a feasible solution
of~$\btheta(H, w)$, then~$(\alpha \lambda, \alpha Z)$ is a feasible
solution of~$\btheta(H, \alpha w)$.  For~(ii), simply take feasible
solutions for~$w$ and~$w'$ and note that their sum is a feasible solution
for~$w + w'$.  For~(iii), we show that if~$w' \leq w$ differs from~$w$ in a
single entry~$x \in V$, then the inequality holds; by applying this result
repeatedly, we then get~(iii).

Indeed, fix~$x \in V$ and let~$(\lambda, Z)$ be an optimal solution
of~$\btheta(H, w)$.  If~$\bar{Z}$ is the Hadamard product of~$Z$
and~$(\one - \chi_{\{x\}})(\one - \chi_{\{x\}})^\tp$,
then~$(\lambda, \bar{Z})$ is a feasible solution of~$\btheta(H, \bar{w})$,
where~$\bar{w}(x) = 0$ and~$\bar{w}(y) = w(y)$ for~$y \neq x$.  By taking
convex combinations of~$(\lambda, \bar{Z})$ and~$(\lambda, Z)$, we then see
that~$\btheta(H, w') \leq \lambda = \btheta(H, w)$ for every~$w'$ such
that~$0 \leq w'(x) \leq w(x)$ and~$w'(y) = w(y)$ for~$y \neq x$.

Back to~(1),
suppose~$\max\{\, w^\tp g : g \in \bthb(H)\,\} < \btheta(H, w)$.
Since~$\bthb(H)$ is a compact set, Theorem~\ref{thm:convex-dual} from
Appendix~\ref{app:duality} gives us a
function~$y\colon \Rplus^V \to \Rplus$, of finite support, such that
\[
  \sum_{\bar{w} \in \Rplus^V} y(\bar{w}) \bar{w} \geq w\quad\text{and}\quad
  \sum_{\bar{w} \in \Rplus^V} y(\bar{w}) \btheta(H, \bar{w}) <
  \btheta(H, w),
\]
and together with~(i),~(ii), and~(iii) we get a contradiction.
  
To see~(2), fix~$l$, $w \in \Rplus^V$ and let~$(\lambda, Z)$ be an optimal
solution of~$\btheta(\overline{H}, w)$.  If $w = 0$, then the result is
immediate, so assume~$w \neq 0$ and
therefore~$\btheta(\overline{H}, w) > 0$.
Then~$\lambda^{-1} Z \in \lthb(H)$ so~$\lambda^{-1} w \in \thb(H)$ and
\[
  \vartheta(H, l) \geq l^\tp (\lambda^{-1} w) =
  \btheta(\overline{H}, w)^{-1} l^\tp w,
\]
proving~(2).

To finish, we prove that if~$f \in \thb(H)$
and~$g \in \bthb(\overline{H})$, then~$f^\tp g \leq 1$, as~(3) then follows
by using Lehman's length-width inequality\footnote{See Theorem~9.5 in
  Schrijver~\cite{Schrijver1986}, where the result is stated for polyhedra,
  but the same proof works for convex sets as well.} together with~(1)
and~(2).  So take~$f \in \thb(H)$ and~$g \in \bthb(\overline{H})$.
Let~$A \in \lthb(H)$ be such that~$\vartheta(H, g) = g^\tp a$,
where~$a = \diag A$.  Note that~$\lambda = 1$ and~$Z = A$ is a feasible
solution of~$\btheta(\overline{H}, a)$,
so~$\btheta(\overline{H}, a) \leq 1$.  Hence
\[
  f^\tp g \leq \vartheta(H, g) = g^\tp a \leq
  \btheta(\overline{H}, a) \leq 1,
\]
and we are done.
\end{proof}

The antiblocker offers another relaxation of the independent-set polytope:
we have the following analogue of Theorem~\ref{thm:theta-basic} and
Corollary~\ref{cor:theta-chi}.

\begin{theorem}
  If~$H$ is an $r$-uniform hypergraph for~$r \geq 2$, then
  \begin{equation}
    \label{eq:thetab-inclusion}
    (r-1)^{-1}\ind(H) \subseteq \bthb(H) \subseteq (r-1)^{-1} \clique(H)
  \end{equation}
  and~$(r-1)^{-1} \alpha(H, w) \leq \btheta(H, w) \leq \chi^*(\overline{H},
  w)$ for every~$w \in \Rplus^V$.
\end{theorem}

\begin{proof}
The antiblocker of~$\ind(H)$ is~$(r-1)^{-1} \clique(\overline{H})$ (see
Theorem~9.4 in Schrijver~\cite{Schrijver1986}).  Since
also~$A(\alpha K) = \alpha^{-1}A(K)$ for every convex set of antiblocking
type~$K$ and~$\alpha > 0$, we get~\eqref{eq:thetab-inclusion} directly from
Theorems~\ref{thm:theta-basic} and~\ref{thm:antiblocker}.

It follows that~$(r-1)^{-1}\alpha(H, w) \leq \btheta(H, w)$.  The proof of
$\btheta(H, w) \leq \chi^*(\overline{H}, w)$ is a straightforward
modification of the proof of Corollary~\ref{cor:theta-chi}.
\end{proof}


\section{Exploiting symmetry}%
\label{sec:symmetry}

When a hypergraph is highly symmetric, the optimization problem over the
theta body or its lifted counterpart can be significantly simplified.  We
enter the realm of invariant semidefinite programs, a topic which has been
thoroughly explored in the last decade~\cite{BachocGSV2012}.  In this
section, we discuss the aspects of the general theory that are most
relevant to our applications.

Let~$V$ be a finite set and let~$\Gamma$ be a finite group that acts
on~$V$.  The action of~$\Gamma$ extends naturally to a function~$f \in
\R^V$: given~$\sigma \in \Gamma$ we define
\[
  (\sigma f)(x) = f(\sigma^{-1}x).
\]
Similarly, the action extends to matrices~$A \in \R^{V \times V}$ by
setting
\[
  (\sigma A)(x, y) = A(\sigma^{-1} x, \sigma^{-1} y)
\]
for every~$\sigma \in \Gamma$.  We say that~$f\in \R^V$ is
\defi{$\Gamma$-invariant} if~$\sigma f = f$ for all~$\sigma \in \Gamma$.
We define $\Gamma$-invariant matrices likewise.

An \defi{automorphism}~$\sigma$ of a hypergraph~$H = (V, E)$ is a
permutation of~$V$ that preserves edges: if~$e \in E$,
then~$\sigma e \in E$.  The set of all automorphisms forms a group under
function composition, called the \defi{automorphism group} of~$H$ and
denoted by~$\aut(H)$.

Let~$H = (V, E)$ be an $r$-uniform hypergraph for~$r \geq 2$.  Consider
first the optimization problem over $\lthb(H)$: given~$w \in \Rplus^V$, we
want to find
\begin{equation}
  \label{opt:theta}
  \max\{\, w^\tp f : \text{$F \in \lthb(H)$ and~$f = \diag F$}\,\}.
\end{equation}
If~$\Gamma \subseteq \aut(H)$ is a group and~$w$ is $\Gamma$-invariant,
then when solving the optimization problem above we may restrict ourselves
to $\Gamma$-invariant matrices~$F$.

Indeed, for~$x \in V$, let~$H_x = (V_x, E_x)$ be the link of~$x$.
Since~$\Gamma \subseteq \aut(H)$, for every~$x \in V$ and
every~$\sigma \in \Gamma$ we have that~$V_{\sigma x} = \sigma V_x$
and~$E_{\sigma x} = \sigma E_x$, hence~$H_{\sigma x} = \sigma H_x$.  It
follows that $\thb(H_{\sigma x}) = \sigma \thb(H_x)$, where the action
of~$\sigma$ maps the function~$f\colon V_x \to \R$ to the
function~$\sigma f\colon V_{\sigma x} \to \R$
by~$(\sigma f)(\sigma y) = f(y)$ for~$y \in V_x$.

This implies that, if~$F \in \lthb(H)$, then~$\sigma F \in \lthb(H)$ for
every~$\sigma \in \Gamma$.  Since~$w$ is invariant, the objective values
of~$F$ and~$\sigma F$ coincide for every~$\sigma \in \Gamma$.  Use the
convexity of $\lthb(H)$ to conclude that, if~$F \in \lthb(H)$, then
\[
  \bar{F} = \frac{1}{|\Gamma|} \sum_{\sigma \in \Gamma} \sigma F
\]
also belongs to~$\lthb(H)$.  Now~$\bar{F}$ is $\Gamma$-invariant and has
the same objective value as~$F$, hence when solving~\eqref{opt:theta} we can
restrict ourselves to $\Gamma$-invariant matrices.

If~$\Gamma$ is a large group, this restriction allows us to
simplify~\eqref{opt:theta} considerably using standard
techniques~\cite{BachocGSV2012}.  The case when~$\Gamma$ acts transitively
on~$V$ is of particular interest to us.

\begin{theorem}%
\label{thm:transitive-theta}
If~$H = (V, E)$ is an $r$-uniform hypergraph for~$r \geq 2$ and
if~$\Gamma \subseteq \aut(H)$ acts transitively on~$V$, then the optimal
value of~\eqref{opt:theta} for~$w = \one$ is equal to the optimal value
of the problem
\begin{equation}
  \label{opt:theta-alt}
  \begin{optprob}
    \max&|V|^{-1} \langle J, A\rangle\\
    &A(x_0, x_0) = 1,\\
    &A_{x_0}[V_{x_0}] \in \thb(H_{x_0}),\\
    &\multicolumn{2}{l}{\text{$A \in \R^{V \times V}$ is positive
        semidefinite and $\Gamma$-invariant,}}
  \end{optprob}
\end{equation}
where~$x_0 \in V$ is any fixed vertex and $J$ is the all-ones matrix.
\end{theorem}

\begin{proof}
Note that~$w = \one$ is $\Gamma$-invariant, so when
solving~\eqref{opt:theta} we can restrict ourselves to $\Gamma$-invariant
matrices.  Let~$F$ be a $\Gamma$-invariant feasible solution
of~\eqref{opt:theta} and set $A = |V| (\one^\tp f)^{-1} F$,
where~$f = \diag F$.

Since~$\Gamma$ acts transitively, all diagonal entries of~$F$ are equal,
hence~$A$ is a feasible solution of~\eqref{opt:theta-alt}.
Now~$\borderM(F)$ is positive semidefinite, and hence the Schur
complement~$F - f f^\tp$ is also positive semidefinite.  So
\[
  |V|^{-1} \langle J, A\rangle = (\one^\tp f)^{-1} \langle J, F\rangle \geq
  (\one^\tp f)^{-1} \langle J, f f^\tp\rangle = \one^\tp f,
\]
and we see that the optimal value of~\eqref{opt:theta-alt} is at least that
of~\eqref{opt:theta}.

For the reverse inequality, let~$A$ be a feasible solution
of~\eqref{opt:theta-alt}.  Since the action of~$\Gamma$ is transitive, we
immediately get that~$A(x, x) = 1$ for all~$x \in V$; it is a little more
involved, though mechanical, to verify that~$A_x[V_x] \in \thb(H_x)$ for
all~$x \in V$.

So set~$F = |V|^{-2} \langle J, A\rangle A$ and~$f = \diag F$; note that
$f = |V|^{-2} \langle J, A\rangle \one$.  Since
$\one^\tp f = |V|^{-1} \langle J, A\rangle$, if we show that~$F$ is a
feasible solution of~\eqref{opt:theta}, then we are done, and to show
that~$F$ is feasible for~\eqref{opt:theta} it suffices to show
that~$\borderM(F)$ is positive semidefinite.

This in turn can be achieved by showing that the Schur
complement~$F - f f^\tp$ is positive semidefinite.  Indeed, note that
since~$A$ is $\Gamma$-invariant, the constant vector~$\one$ is an
eigenvector of~$A$ with eigenvalue~$|V|^{-1} \langle J, A\rangle$.
Hence~$\one$ is an eigenvector of both~$F$ and~$f f^\tp$ with the same
eigenvalue; since all other eigenvalues of $f f^\tp$ are zero and $F$ is
positive semidefinite, we are done.
\end{proof}

Symmetry also simplifies testing whether a given vector is in the theta
body.

\begin{theorem}%
  \label{thm:invariant-external}
  Let~$H = (V, E)$ be an $r$-uniform hypergraph with~$r \geq 2$ and
  let~$\Gamma \subseteq \aut(H)$ be a group.  A $\Gamma$-invariant
  vector~$f \in \R^V$ is in~$\thb(H)$ if and only if~$f \geq 0$ and~$w^\tp
  f \leq \vartheta(H, w)$ for every $\Gamma$-invariant~$w \in \Rplus^V$.
\end{theorem}

\begin{proof}
Necessity being trivial from Theorem~\ref{thm:theta-antiblocking}, let us
prove sufficiency.  If~$w\in \Rplus^V$ is any weight function, then
since~$f$ is $\Gamma$-invariant we have that
\[
  w^\tp f = \frac{1}{|\Gamma|} \sum_{\sigma \in \Gamma} w^\tp (\sigma f) =
  \frac{1}{|\Gamma|} \sum_{\sigma \in \Gamma} (\sigma^{-1}w)^\tp f =
  \bar{w}^\tp f,
\]
where~$\bar{w} = |\Gamma|^{-1}\sum_{\sigma\in\Gamma} \sigma^{-1} w$.  Note
that~$\bar{w}$ is $\Gamma$-invariant.

We claim that~$\vartheta(H, \bar{w}) \leq \vartheta(H, w)$.  Indeed,
since~$\bar{w}$ is $\Gamma$-invariant, let~$g \in \thb(H)$ be a
$\Gamma$-invariant vector such that~$\bar{w}^\tp g = \vartheta(H,
\bar{w})$.  Then~$(\sigma w)^\tp g =  w^\tp (\sigma^{-1} g) = w^\tp g$ for
all~$\sigma \in \Gamma$, and so~$w^\tp g = \bar{w}^\tp g$,
hence~$\vartheta(H, w) \geq \bar{w}^\tp g$, proving the claim.

Now use to claim to get
$w^\tp f = \bar{w}^\tp f \leq \vartheta(H, \bar{w}) \leq \vartheta(H, w)$,
and with Theorem~\ref{thm:theta-antiblocking} we are done.
\end{proof}


\section{Triangle-encoding hypergraphs and Mantel's theorem}%
\label{sec:mantel}

In a~1910 issue of the journal \textit{Wiskundige Opgaven,} published by
the Royal Dutch Mathematical Society, Mantel~\cite{Mantel1910} asked what
perhaps turned out to be the first question of extremal graph theory; in
modern terminology: how many edges can a triangle-free graph on~$n$
vertices have?  The complete bipartite graph on~$n$ vertices with parts of
size~$\lfloor n/2\rfloor$ and~$\lceil n/2\rceil$ is triangle-free and has
$\lfloor n^2/4\rfloor$ edges.  Can we do better?

The answer came in the same issue, supplied by Mantel and several others:
\textit{a triangle-free graph on~$n$ vertices has at
  most~$\lfloor n^2/4\rfloor$ edges.}  Mantel's theorem was later
generalized by Turán to $K_r$-free graphs for~$r \geq 4$.

Given an integer~$n \geq 3$, we want to find the largest triangle-free
graph on~$[n] = \{1, \ldots, n\}$.  So we construct a $3$-uniform
hypergraph~$H_n = (V_n, E_n)$ as follows:
\begin{itemize}
\item the vertices of~$H_n$ are the edges of the complete graph~$K_n$
  on~$[n]$;
  
\item three vertices of~$H_n$, corresponding to three edges of~$K_n$,
  form an edge of~$H_n$ if they form a triangle in~$K_n$.
\end{itemize}
Independent sets of~$H_n$ thus correspond to triangle-free subgraphs
of~$K_n$, and the independent-set polytope of~$H_n$ coincides with the
Turán polytope studied by Raymond~\cite{Raymond2018}.  In order to
illustrate our methods we will compute the theta number~$\vartheta(H_n)$,
which provides an upper bound of~$n^2/4$ for the independence number
of~$H_n$.  This bound, rounded down, coincides with the lower bound
$\lfloor n^2/4\rfloor$ given by the complete bipartite graph, showing that
the theta number is essentially tight for this infinite family of
hypergraphs.  Incidentally, this gives another proof of Mantel's theorem,
though not a particularly short one.

The symmetric group~$\Scal_n$ on~$n$ elements acts on~$[n]$, and therefore
on~$V_n$, and this action preserves edges of~$H_n$, hence~$\Scal_n$ is a
subgroup of~$\aut(H_n)$.  The action of~$\Scal_n$ is also transitive, so we
set~$x_0 = \{1, 2\}$ and use Theorem~\ref{thm:transitive-theta} to get
\begin{equation}
  \label{opt:mantel-theta}
  \begin{optprob}
    \vartheta(H_n) = \max&|V_n|^{-1} \langle J, A\rangle\\
    &A(x_0, x_0) = 1,\\
    &A_{x_0}[(V_n)_{x_0}] \in \thb((H_n)_{x_0}),\\
    &\multicolumn{2}{l}{\text{$A \in \R^{V_n \times V_n}$ is positive
        semidefinite and $\Scal_n$-invariant.}}
  \end{optprob}
\end{equation}

The link of~$x_0 = \{1,2\}$ is the graph with vertex set
\[
  (V_n)_{x_0} = \{\, \{1, k\} : k \in \{3, \ldots, n\}\,\}
  \cup \{\, \{2, k\} : k \in \{3, \ldots, n\}\,\}
\]
and edge set
\[
  (E_n)_{x_0} = \{\,\{\{1, k\}, \{2,k\}\} : k \in \{3,
  \ldots, n\}\,\},
\]
that is, it is a matching with~$2(n-2)$ vertices and~$n-2$ edges (see
Figure~\ref{fig:mantel}).

\begin{figure}[t]
  \begin{center}
    \includegraphics{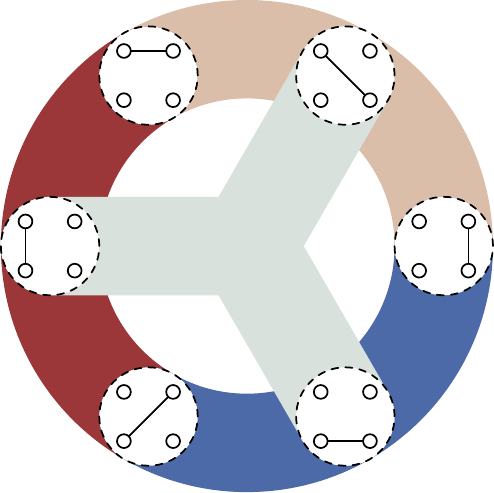}\hspace{1.5cm}
    \includegraphics{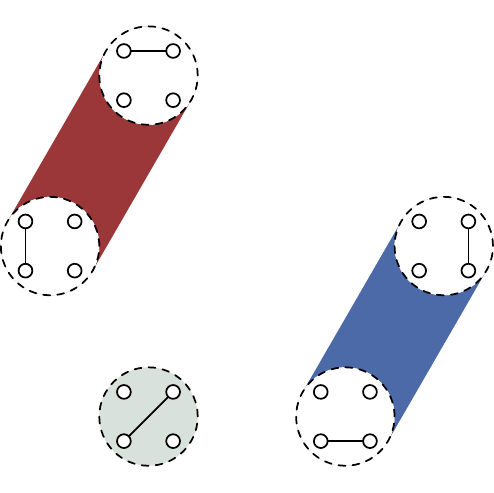}
  \end{center}

  \caption{On the left, the hypergraph~$H_4$ where each vertex is the edge
    of~$K_4$ shown.  On the right, the link
    of~\lower3pt\hbox{\includegraphics{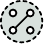}}, consisting of a
    matching with~4 vertices and~2 edges.}%
  \label{fig:mantel}
\end{figure}

The row~$A_{x_0}$ of an~$\Scal_n$-invariant
matrix~$A \in \R^{V_n \times V_n}$ is invariant under the stabilizer
of~$x_0$, and so~$A_{x_0}[V_{x_0}]$ is a constant function since the
stabilizer acts transitively on~$(V_n)_{x_0}$.
Theorem~\ref{thm:invariant-external} then implies
that~$A_{x_0}[(V_n)_{x_0}] \in \thb((H_n)_{x_0})$ if and only if
\begin{equation}%
  \label{eq:alt-mantel-link}
  0 \leq A(x_0, \{1, 3\}) \leq \frac{\vartheta((H_n)_{x_0})}{2(n-2)}  =
  \frac{1}{2},
\end{equation}
since~$n-2 \leq \alpha((H_n)_{x_0}) \leq \vartheta((H_n)_{x_0}) \leq
\chi(\overline{(H_n)_{x_0}}) \leq n-2$.

We simplify this problem further by computing a basis of the space of
$\Scal_n$-invariant symmetric matrices in~$\R^{V_n \times V_n}$.  The
action of~$\Scal_n$ on~$V_n$ extends naturally to an action
on~$V_n \times V_n$.  There are three orbits of~$V_n \times V_n$ under this
action, namely~$R_k = \{\, (x, y) : |x \cap y| = 2-k\, \}$ for~$k = 0$,
$1$, and~$2$.  So a basis of the invariant subspace is given by the
matrices~$A_k$ such that
\[
  A_k(x, y) = \begin{cases}
    1,&\text{if~$(x, y) \in R_k$;}\\
    0,&\text{otherwise.}
  \end{cases}
\]
Note that~$A_0$ is the identity matrix.

A feasible solution of~\eqref{opt:mantel-theta} is then of the form
\begin{equation}
  \label{eq:mantel-solution}
  A = A_0 + \alpha A_1 + \beta A_2
\end{equation}
for some real numbers~$\alpha$ and~$\beta$.  We see moreover
that~$A(x_0, \{1,3\}) = \alpha$, and so~\eqref{eq:alt-mantel-link}
becomes~$0 \leq \alpha \leq 1/2$.  The objective function is
\begin{equation}%
  \label{eq:mantel-obj}
  \begin{split}
    |V_n|^{-1} \langle J, A\rangle &= |V_n|^{-1} (\langle J, A_0\rangle +
    \alpha \langle J, A_1 \rangle + \beta \langle J, A_2\rangle)\\
    &= 1 + |V_n|^{-1} |R_1| \alpha + |V_n|^{-1} |R_2| \beta.
  \end{split}
\end{equation}

For the positive semidefiniteness constraint on~$A$, we observe
that~$\{A_0, A_1, A_2\}$ is the Johnson scheme~$\Jcal(n, 2)$ (see Godsil
and Meagher~\cite[Chapter~6]{GodsilM2016}).  The algebra spanned by the
scheme (its \defi{Bose-Mesner algebra}) is commutative, unital, and closed
under transposition; its matrices then share a common basis of
eigenvectors, say~$v_1$, $v_2$, and~$v_3$, and can therefore be
simultaneously diagonalized.  The eigenvalues of~$v_1$, $v_2$, and~$v_3$
for each matrix are (see Theorem~6.5.2, ibid.):
\[
  \begin{array}{ll}
    A_0\colon&1,\ 1,\ 1;\\
    A_1\colon&-2,\ n-4,\ 2n-4;\\
    A_2\colon&1,\ -(n-3),\ (n-2)(n-3)/2.
\end{array}
\]

Putting it all together, our original problem can be rewritten as
\[
  \begin{optprob}
    \max&1 + |V_n|^{-1} |R_1| \alpha + |V_n|^{-1} |R_2| \beta\\
    &0 \leq \alpha \leq 1/2,\\
    &1 -2\alpha + \beta \geq 0,\\
    &1 +(n-4)\alpha -(n-3)\beta \geq 0,\\
    &1 +(2n-4)\alpha + ((n-2)(n-3)/2)\beta \geq 0.
  \end{optprob}
\]
This is a linear program on two variables.  Using the dual, or finding all
vertices of the primal feasible region, it is easy to verify that one
optimal solution is
\[
  \alpha = 1/2\quad\text{and}\quad \beta = \frac{n-2}{2(n-3)}
\]
for all~$n \geq 4$.  This gives us an optimal value of~$n^2/4$, which
rounded down coincides with the lower bound coming from complete bipartite
graphs.


\section{Triangle-avoiding sets in the Hamming cube}%
\label{sec:hamming}

For an integer~$n \geq 1$, consider the Hamming cube~$\Ham^n = \{0,1\}^n$
equipped with the \defi{Hamming distance}, which for~$x$, $y \in \Ham^n$ is
denoted by~$d(x, y)$ and equals the number of bits in which~$x$ and~$y$
differ.  A classical problem in coding theory is to determine the
parameter~$A(n, d)$, which is the maximum size of a subset~$I$ of~$\Ham^n$
such that~$d(x, y) \geq d$ for all distinct~$x$, $y \in I$.

If we let~$G(n, d)$ be the graph with vertex set~$\Ham^n$ in which~$x$,
$y \in \Ham^n$ are adjacent if~$d(x, y) < d$,
then~$A(n, d) = \alpha(G(n, d))$.  A simple variant of the Lovász theta
number of~$G(n, d)$, obtained by requiring that~$F$
in~\eqref{eq:theta-graph} be nonnegative as well, then provides an upper
bound for~$A(n, d)$, which is easy to compute given the abundant symmetry
of~$G(n, d)$.  This bound, known as the linear programming bound, was
originally described by Delsarte~\cite{Delsarte1973}; its relation to the
theta number was later discovered by McEliece, Rodemich, and
Rumsey~\cite{McElieceRR1978} and Schrijver~\cite{Schrijver1979}.

We now consider a hypergraph analogue of this problem.  Let~$s \geq 1$ be
an integer.  Three distinct points~$x_1$, $x_2$, $x_3 \in \Ham^n$ form an
\defi{$s$-triangle} if~$d(x_i, x_j) = s$ for all~$i \neq j$.  It is easy to
show that there is an $s$-triangle in $\Ham^n$ if and only if~$s$ is even
and~$0 < s \leq \lfloor 2n/3\rfloor$.

We want to find the largest size of a set of points in~$\Ham^n$ that avoids
$s$-triangles.  More precisely, given integers~$n$, $s \geq 1$, we consider
the hypergraph~$H(n, s)$ whose vertex set is~$\Ham^n$ and whose edges are
all $s$-triangles and we want to find its independence number.  The theta
number~$\vartheta(H(n, s))$ defined in~\eqref{eq:theta-number} gives us an
upper bound.

To compute~$\vartheta(H(n, s))$, start by noting that~$\iso(\Ham^n)$, the
group of isometries of~$\Ham^n$, is a subgroup of the automorphism group
of~$H = H(n, s)$ and, since it acts transitively on~$\Ham^n$, we can use
Theorem~\ref{thm:transitive-theta} to simplify our problem.  To do so we
choose~$x_0 = 0$.

The vertex set of the link~$H_0$ of~$0$ is~$\Ham^n_s$, the set of all words
of weight~$s$, the \defi{weight} of a word being the number of~1s in it;
two words are adjacent in~$H_0$ if they are at distance~$s$.  The isometry
group~$\iso(\Ham^n_s)$ of~$\Ham^n_s$ is a subgroup of the automorphism
group of~$H_0$.

If~$A\colon \Ham^n \times \Ham^n \to \R$ is an $\iso(\Ham^n)$-invariant
symmetric matrix, then~$A(x, y)$ depends only on~$d(x, y)$, and
so~$a = A_0[V_0]$ is a constant function.  We write~$A(t)$ for the value
of~$A(x, y)$ when~$d(x, y) = t$.

By Theorem~\ref{thm:invariant-external}, we have~$a \in \thb(H_0)$ if and
only if~$a \geq 0$ and $w^\tp a \leq \vartheta(H_0, w)$ for every
$\iso(\Ham^n_s)$-invariant~$w \in \Rplus^{V_0}$.  Since $\iso(\Ham^n_s)$
acts transitively on~$\Ham^n_s$, every such invariant~$w$ is constant, and
we conclude that $A_0[V_0] \in \thb(H_0)$ if and only if
$0 \leq |\Ham^n_s| A(s) \leq \vartheta(H_0)$.

The problem can be further simplified.  A
matrix~$A\colon \Ham^n \times \Ham^n \to \R$ is $\iso(\Ham^n)$-invariant
and positive semidefinite if and only if there are numbers~$a_0$,
\dots,~$a_n \geq 0$ such that
\[
  A(t) = \sum_{k=0}^n a_k K^n_k(t),
\]
where~$K^n_k$ is the Krawtchouk polynomial of degree~$k$, normalized
so~$K^n_k(0) = 1$.  This polynomial can be defined on
integers~$t \in \{0, \ldots,n\}$ by the formula
\[
  K^n_k(t) = \binom{n}{k}^{-1} \sum_{i=0}^k (-1)^i \binom{t}{i}
  \binom{n-t}{k-i}.
\]
If~$E_k(x, y) = K^n_k(d(x, y))$, then we have the orthogonality relations
$\langle E_k, E_l\rangle = 0$ for~$k \neq l$; see~Dunkl~\cite{Dunkl1976}.

With this characterization, and noting that~$E_0 = J$ is the all-ones
matrix, we have
\[
\langle J, A\rangle = \langle J, a_0 E_0 \rangle = |\Ham^n|^2 a_0 = 2^{2n} a_0.
\]
Rewriting~\eqref{opt:theta-alt}, we see that~$\vartheta(H(n,s))$ is the
optimal value of the problem
\begin{equation}%
  \label{opt:ham-triangle}
  \begin{optprob}
    \max&2^n a_0\\
    &\sum_{k=0}^n a_k = 1,\\
    &\sum_{k=0}^n a_k K^n_k(s) \leq |\Ham^n_s|^{-1} \vartheta(H_0),\\
    &\text{$a_0$, \dots,~$a_n \geq 0$.}
  \end{optprob}
\end{equation}
Here, we have omitted the constraint~$0 \leq |\Ham_s^n| A(s)$, since it is
automatically satisfied by the optimal solution.

Problem~\eqref{opt:ham-triangle} has only two constraints, and so its
optimal solution admits a simple expression.
With~$M_K^n(s) = \min\{\, K^n_k(s) : \text{$k = 0$, \dots,~$n$}\,\}$
for~$s \geq 0$ we have:

\begin{theorem}%
  \label{thm:theta-expr}%
  If~$n \geq 1$ is an integer and~$0 < s \leq \lfloor 2n/3\rfloor$ is an
  even integer, then
  \begin{equation}%
    \label{eq:theta-expr}
    \vartheta(H(n, s)) = 2^n\frac{M_K^n(s) - |\Ham^n_s|^{-1}
      \vartheta(H(n, s)_0)}{M_K^n(s) - 1}.
  \end{equation}
\end{theorem}

\begin{proof}
Write~$H = H(n, s)$.  By our choice of~$s$, there are $s$-triangles
in~$\Ham^n$, so~$H_0$ is a nonempty graph.
Hence~$\vartheta(H_0) \leq \chi(\overline{H_0}) < |\Ham^n_s|$, and so a
feasible solution of~\eqref{opt:ham-triangle} has to use some
variable~$a_k$ for~$k > 0$.

To solve our problem we want to maximize~$a_0$ keeping the convex
combination
\[
  \sum_{k=0}^n a_k K^n_k(s)
\]
below~$|\Ham^n_s|^{-1} \vartheta(H_0)$.  We cannot achieve this by using
only~$a_0$, so the best way to do it is to let~$k^*$ be such
that~$K^n_{k^*}(s) = M_K^n(s)$ and use only the variables~$a_0$
and~$a_{k^*}$.  This leads us to the system
\begin{align*}
  a_0 + a_{k^*} &= 1,\\
  a_0 + a_{k^*} M_K^n(s) &= |\Ham^n_s|^{-1} \vartheta(H_0),
\end{align*}
whose solution yields exactly~\eqref{eq:theta-expr}.
\end{proof}

To compute~$\vartheta(H_0)$ we again use symmetry.
Let~$A\colon \Ham^n_s \times \Ham^n_s \to \R$ be a matrix.  If~$A$ is
$\iso(\Ham^n_s)$-invariant, then~$A(x, y)$ depends only on~$d(x, y)$, and
so we write~$A(t)$ for the value of~$A(x, y)$ when~$d(x, y) = t$.  The
matrix~$A$ is $\iso(\Ham^n_s)$-invariant and positive semidefinite if and
only if there are numbers~$a_0$, \dots,~$a_s \geq 0$ such that
\[
  A(t) = \sum_{k=0}^s a_k Q^{n,s}_k(t/2)
\]
(note that Hamming distances in~$\Ham^n_s$ are always even),
where~$Q^{n,s}_k$ is the Hahn polynomial of degree~$k$, normalized
so~$Q_k^{n,s}(0) = 1$.  For an integer~$0 \leq t \leq s$, these polynomials
are given by the formula
\[
  Q^{n,s}_k(t) = \sum_{i=0}^k (-1)^i \binom{s}{i}^{-1} \binom{n-s}{i}^{-1}
  \binom{k}{i} \binom{n+1-k}{i} \binom{t}{i}.
\]
If~$E_k(x, y) = Q^{n,s}_k(d(x, y) / 2)$, then~$\langle E_k, E_l\rangle = 0$
whenever~$k \neq l$ (see Delsarte~\cite{Delsarte1978}, in particular
Theorem~5, and Dunkl~\cite{Dunkl1978}).

With this characterization, $\langle J, A\rangle = |\Ham^n_s|^2 a_0$
since~$E_0 = J$.  Rewriting~\eqref{opt:theta-alt}, we see
that~$\vartheta(H_0)$ is the optimal value of the problem
\begin{equation}%
  \label{opt:theta-link}
  \begin{optprob}
    \max&|\Ham^n_s| a_0\\
    &\sum_{k=0}^s a_k = 1,\\
    &\sum_{k=0}^s a_k Q^{n,s}_k(s/2) = 0,\\
    &\text{$a_0$, \dots,~$a_s \geq 0$.}
  \end{optprob}
\end{equation}

Writing $M^n_Q(s) = \min\{\, Q^{n,s}_k(s/2) : \text{$k=0$,
  \dots,~$s$}\,\}$, we have the analogue of Theorem~\ref{thm:theta-expr}.

\begin{theorem}%
\label{thm:theta-link-expr}
If~$n \geq 1$ is an integer and~$0 < s \leq \lfloor 2n/3\rfloor$ is an even
integer, then
\[
  \vartheta(H(n, s)_0) = |\Ham^n_s| \frac{M_Q^n(s)}{M_Q^n(s) - 1}.
\]
\end{theorem}

\begin{proof}
  Adapt the proof of Theorem~\ref{thm:theta-expr}.
\end{proof}

The upshot is that $\vartheta(H(n,s))$ may be expressed entirely in terms
of the parameters
\begin{align}
    \label{eq:MK}
    M_K^n(s) &= \min\{\, K^n_k(s) : k = 0, \dots,n \,\}\quad\text{and}\\
    \label{eq:MQ}
    M^n_Q(s) &= \min\{\, Q^{n,s}_k(s/2) : k=0, \dots,s \,\}.
\end{align}
Very similar expressions can be derived for the theta number in the more
general setting of \defi{$q$-ary cubes} $\{0, \dots, q-1\}^n$ for any
integer~$q \geq 2$; in this case we must use Krawtchouk polynomials with
weight~$(q-1)/q$ (see Dunkl~\cite{Dunkl1976}) and $q$-ary Hahn
polynomials~\cite{Delsarte1978}.

The theta number for hypergraphs can also be extended to some well-behaved
infinite hypergraphs, and can be used in particular to provide upper bounds
for the density of simplex-avoiding sets on the sphere and in Euclidean
space~\cite{CastroOSV2022}.  For triangle-avoiding sets on the
sphere~$S^{n-1}$, for instance, the bound obtained is like the one in
Theorems~\ref{thm:theta-expr} and~\ref{thm:theta-link-expr}, with both the
Krawtchouk and Hahn polynomials replaced by Gegenbauer (ultraspherical)
polynomials $P^n_k$ (resp.~$P^{n-1}_k$), which are the orthogonal
polynomials on the interval $[-1, 1]$ for the weight function
$(1-x^2)^{(n-3)/2}$.  In this setting, the link of a vertex~$x \in S^{n-1}$
is a scaled copy of~$S^{n-2}$.

This bound can be analyzed asymptotically, yielding an upper bound for
the density of simplex-avoiding sets that decays exponentially fast in
the dimension of the underlying space. The key point in the analysis
is to show exponential decay of the parameter
$M_P^n(t) = \min\{\, P^n_k(t) : k \geq 0 \,\}$ for $t \in (0,1)$. This
is done in two steps. First, one uses results on the asymptotic
behavior of the roots of Gegenbauer polynomials to show that
$\min\{\, P^n_k(t) : k \geq 0 \,\}$ is attained at~$k =
\Omega(n)$. Then, one shows that $|P^n_k(t)|$ tends to~$0$
exponentially fast if $k = \Omega(n)$ by exploiting a particular
integral representation for the Gegenbauer
polynomials~\cite[Lemma~4.2]{CastroOSV2022}.

\begin{figure}[t]
  \begin{center}
    \includegraphics{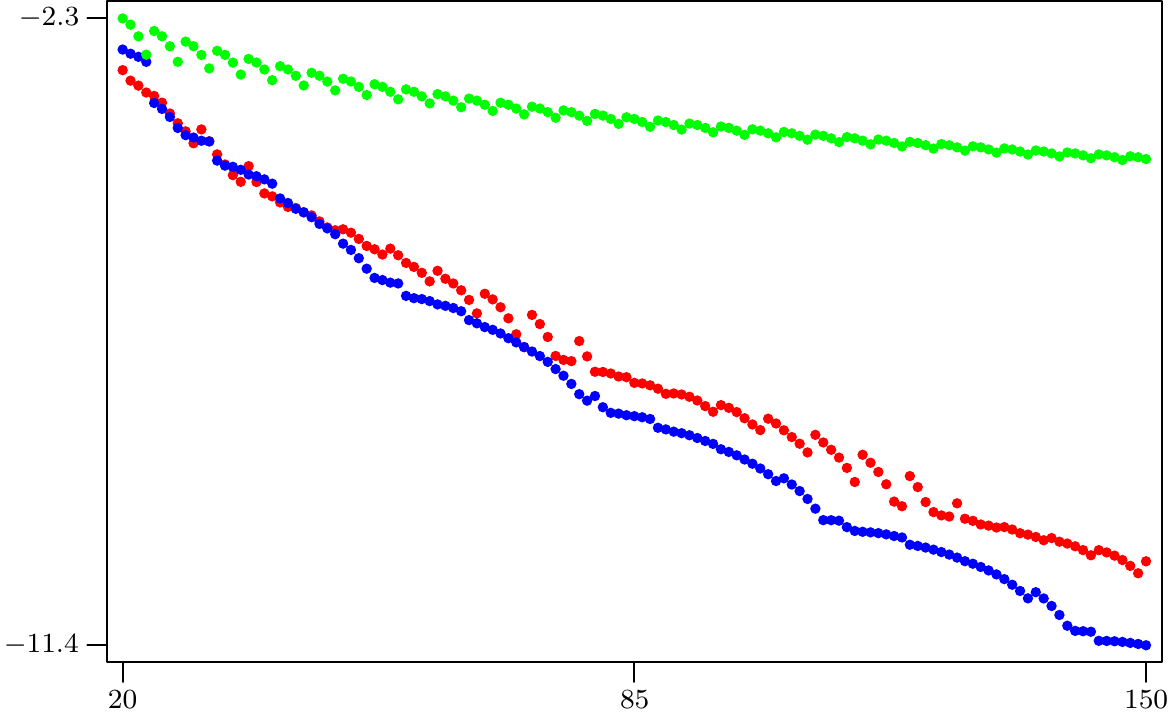}
  \end{center}

  \caption{The plot shows, for every~$n = 20$, \dots,~$150$ on the
    horizontal axis, the value of~$\ln(\vartheta(H(n, s)) / 2^n)$ on the
    vertical axis, where~$s$ is the even integer closest to~$n/2$ (in
    green), $n/3$ (in red), and~$n/4$ (in blue).}%
  \label{fig:decay}
\end{figure}

The same can be attempted for the Hamming cube: how does the density of a
subset of~$\Ham^n$ that avoids $s$-triangles behave as~$n$ goes to
infinity?  For a fixed~$s$, the answer is simple, since~$|\Ham^n_s|$ is
exponentially smaller than~$|\Ham^n|$. We should therefore consider a
regime where $s$ and $n$ both tend to infinity; for instance, we could
take~$s = s(n, c)$ to be the even integer closest to~$n/c$ for
some~$c > 1$.  Numerical evidence (see Figure~\ref{fig:decay}) supports the
following conjecture.

\begin{conjecture}%
  \label{conj:hamming}
  With~$s(n, c)$ defined as above,~$\vartheta(H(n, s(n, c))) / 2^n$
  decays exponentially fast with~$n$ for every fixed~$c > 2$,
  whereas~$\vartheta(H(n, s(n, 2))) / 2^n$ decays linearly fast
  with~$n$.
\end{conjecture}

We leave open the question of whether this conjecture, for~$c > 2$,
can be proven using Theorems~\ref{thm:theta-expr}
and~\ref{thm:theta-link-expr}. Following the strategy of Castro-Silva,
Oliveira, Slot, and Vallentin~\cite{CastroOSV2022}, it is possible to
show that the minima in~\eqref{eq:MK} and~\eqref{eq:MQ} are attained
at~$k = \Omega(n)$, using results on the roots of Krawtchouk and Hahn
polynomials.  For~$c = 2$, it appears that the minimum
in~\eqref{eq:MK} is always attained at~$k = 2$ when~$n$ is a multiple
of~4, implying in this case that $M_n(n/2) = K_2^n(n/2) = -1/(n-1)$.
The remaining obstacle to finishing the analysis of the asymptotic
behavior of $M_K^n(s)$ and $M_Q^n(s)$ is the lack of a suitable
integral representation for the Krawtchouk and Hahn polynomials, as
was available for the Gegenbauer polynomials.


%
%

\appendix

\section{Duality for optimization over compact convex sets}%
\label{app:duality}

In the proof of Theorem~\ref{thm:antiblocker} we used a kind of linear
programming dual of the optimization problem over~$\bthb(H)$ to get a
contradiction, but~$\bthb(H)$ is not necessarily a polyhedron.  Since it is
convex and compact, however, a kind of strong duality holds.  The following
theorem should be known, but we could not find a suitable reference.

\begin{theorem}%
\label{thm:convex-dual}%
Let~$I$ be a set and for every~$i \in I$ let~$a_i \in \R^n$ and~$\beta_i
\in \R$; write~$S = \{\, x \in \Rplus^n : \text{$a_i^\tp x \leq \beta_i$
  for~$i \in I$}\,\}$.  If~$S$ is nonempty and compact, then for every~$c
\in \R^n$ we have that~$\max\{\, c^\tp x : x \in S\,\}$ is the optimal
value of 
\begin{equation}%
  \label{eq:convex-dual}%
  \begin{optprob}
    \inf&\sum_{i \in I} y_i \beta_i\\
    &\sum_{i \in I} y_i a_i \geq c,\\
    &\text{$y\colon I \to \Rplus$ has finite support.}
  \end{optprob}
\end{equation}
\end{theorem}

\begin{proof}
It is easy to show that~$\max \leq \inf$: if~$x \in S$ and if~$y$ is a
feasible solution of~\eqref{eq:convex-dual}, then
\[
  c^\tp x \leq \biggl(\sum_{i \in I} y_i a_i\biggr)^\tp x = \sum_{i \in I}
  y_i a_i^\tp x \leq \sum_{i \in I} y_i \beta_i.
\]

For the reverse inequality, start by observing that we may assume that~$I$
is countable.  Indeed, center on each rational point in~$\R^{n+1}$ balls of
radii~$1/k$ for each integer~$k \geq 1$.  Inside every such ball choose a
point~$(a_i, \beta_i)$, for~$i \in I$, if such a point exists.  This gives
a countable subset of~$I$ defining the same set~$S$.

So say~$I = \{1, 2, \ldots\}$ and for an integer~$k \geq 1$ write
\[
  S_k = \{\, x \in \Rplus^n : \text{$a_i^\tp x \leq \beta_i$ for~$1 \leq i
    \leq k$}\,\}.
\]
We claim that there is~$k_0$ such that~$S_{k_0}$ is compact.  If not, then
for every~$k \geq 1$ there is a nonzero~$z_k \in \Rplus^n$ such
that~$a_i^\tp z_k \leq 0$ for all~$1 \leq i \leq k$.  If we normalize these
points so~$\|z_k\| = 1$ for every~$k$, then the sequence~$(z_k)$ has a
converging subsequence, and we may assume that the sequence itself
converges, say to a point~$z$ with~$\|z\| = 1$.  Note that~$z \geq 0$.
Moreover, for every~$i \geq 1$ we have
\[
  a_i^\tp z = \lim_{k\to\infty} a_i^\tp z_k \leq 0,
\]
and since~$S$ is nonempty it follows that~$S$ is unbounded, a
contradiction.

So for every~$k \geq k_0$ let~$x_k^*$ be an optimal solution
of~$\max\{\, c^\tp x : x \in S_k\,\}$.  Since~$S_{k_0}$ is bounded
and~$x^*_k \in S_{k_0}$ for every~$k \geq k_0$, the sequence~$(x^*_k)$ has
a converging subsequence; assume the sequence itself converges to~$x^*$.
Then~$x^* \geq 0$ and for~$i \geq 1$ we have
\[
  a_i^\tp x^* = \lim_{k \to \infty} a_i^\tp x^*_k \leq \beta_i,
\]
so~$x^* \in S$.  Moreover, since
$c^\tp x_k^* \geq \max\{\, c^\tp x : x \in S\,\}$ for every~$k$, we
conclude that~$x^*$ is an optimal solution
of~$\max\{\, c^\tp x : x \in S\,\}$.

The strong duality theorem of linear programming gives us, for
every~$k \geq k_0$, a function~$y_k\colon I \to \Rplus$, supported
on~$[k]$, such that
\[
  \sum_{i \in I} (y_k)_i a_i \geq c\quad\text{and}\quad
  \sum_{i \in I} (y_k)_i \beta_i = c^\tp x_k^*.
\]
Each~$y_k$ is a feasible solution of~\eqref{eq:convex-dual}, and it follows
that the optimal value of~\eqref{eq:convex-dual} is~$\leq c^\tp x^*$.
\end{proof}

\end{document}